\documentclass[8pt,a4paper]{article}
\usepackage[utf8]{inputenc}
\usepackage{amsfonts}
\usepackage{amsthm}
\usepackage{amsmath}
\usepackage{amssymb}
\usepackage{enumitem,kantlipsum}

\usepackage{graphicx}
\usepackage{mathrsfs}
\usepackage{calligra}
\usepackage{xcolor}

\newtheorem{theorem}{Theorem}[subsection]
\theoremstyle{definition}
\newtheorem{definition}[theorem]{Definition}
\newtheorem{cor}[theorem]{Corollary}
\newtheorem{lemma}[theorem]{Lemma}
\newtheorem{proposition}[theorem]{Proposition}

\newtheorem{example}[theorem]{Example}
\newtheorem{remark}[theorem]{Remark}

\title{Simply-laced mixed-sign Coxeter groups, with an associate graph is a line or a simple cycle}
\author{Yiska Efrat Aharoni, Robert Shwartz \\
Department of Mathematics \\
Ariel University.}

\begin{document}

\author{Yiska Efrat Aharoni\\Department of Mathematics\\Ariel University, Israel\\yiskaefr.aharoni@msmail.ariel.ac.il\and Robert Shwartz\\Department of Mathematics\\Ariel University, Israel\\robertsh@ariel.ac.il}
\date{}
\title{Simply-laced mixed-sign Coxeter groups, with an associate graph is a line or a simple cycle}
\maketitle

\begin{abstract}

   In 2011 Eriko Hironaka introduced an interesting generalization of Coxeter groups, motivated by studying certain mapping classes. The generalization is by labeling the vertices of a Coxeter graph either by $+1$ or by $-1$, and then generalizing the standard geometric representation of the associated Coxeter group by concerning the labels of the vertices. The group which Hironaka get by that generalization is called mixed-sign Coxeter group. In this paper we classify the simply-laced mixed-sign Coxeter groups where the associated graph is either a line or a simple cycle. We show that all  the defining relations of the mixed-sign Coxeter groups  with the mentioned associated graph (either a line or a simple cycle) are squares or cubes of a product of conjugates of two generators of the mixed-sign Coxeter group and are strongly connected to the labels of the vertices of the associated graph.
\end{abstract}

\section{Introduction}\label{Introduction}

\subsection{Coxeter groups}

Coxeter Groups are an important class of groups which is widely  used in a lot of fields and a lot of aspects of mathematics, like the study of
symmetries and reflections, classifications of Lie Algebras and in a lot of other subjects. Hence, we start with recalling the definition of Coxeter groups, and some basic related concepts and properties of it, as it is defined in \cite{BOOK}.

\begin{definition}\label{cox-group}
A group $W$ is a Coxter group if $W$ has the following presentation in terms of generators and relations:
$$W=\langle s_1, s_2, \ldots, s_n ~|~ s_i^2=1, ~\left(s_i s_j\right)^{m_{i,~j}}=1, ~1\leq i, j\leq n, ~m_{i,~j}\in \mathbb{N}, ~m_{i,~j}\geq 2\rangle.$$
i.e.,
\begin{enumerate}
    \item $W$ is generated by $s_1, s_2, \ldots, s_n$, where $s_i$ is an involution (element of order $2$) for all \\$1\leq i\leq n$;
    \item All the defining relations of $W$ has the form:
    $$s_i^2=1,\quad  (s_is_j)^{m_{i,~j}}=1,$$ where $i, j\in \{1,2...n\},$ $i\neq j$ and  $m_{i,~j}$ is the order of the element $s_is_j$ in $W$.
\end{enumerate} 
\end{definition}
\begin{definition}\label{cox-matrix}
Let $S$ be a set of generators of a Coxeter group $W$. 
\begin{itemize}
    \item The associated matrix $M: S \times S \rightarrow\{1,2, \ldots, \infty\}$, which is called the Coxeter matrix of the group $W$, satisfies the following properties:
    
    \begin{itemize}
        \item $m_{s, s^{\prime}} =m_{s^{\prime}, s}=|s\cdot s^{\prime}|$ ~(where, $|s\cdot s^{\prime}|$ is the order of the element $s\cdot s^{\prime}$ in the group $W$); 
        \item $m_{s, s^{\prime}} =1 \Leftrightarrow s=s^{\prime}$.
    \end{itemize}

\item The associated Coxeter graph is a graph with vertices in $S$ satisfies the following properties: 
\begin{itemize}
    \item If $m_{s, s^{\prime}}=2$, then the two  vertices $s$ and $s^{\prime}$ are not connected by any edge in the Coxeter graph;
    \item If $m_{s, s^{\prime}}=3$, then there is an unlabeled edge which connects the vertices  $s, s^{\prime}$ which is called a simply-laced edge;
    \item If $m_{s, s^{\prime}}\geq 4$, then there is an  edge which is labeled by $m_{s, s^{\prime}}$ that  connects the vertices  $s, s^{\prime}$;
    \item A Coxeter graph is called simply-laced Coxeter graph if all the edges are simply-laced. Then the associated Coxeter group is called simply-laced Coxeter group (i.e., there is satisfied $m_{s_i, s_j}\in \{2, 3\}$  for all $1\leq i, j\leq n$).
\end{itemize}
\end{itemize}
\end{definition}

\begin{remark}
Let $G$ be an edge labeled graph without loops, then there exists a unique  Coxeter group denoted by $W(G)$ such that $G$ is the Coxeter graph of $W(G)$.
\end{remark}

\begin{remark}
Notice, the following properties:
\begin{itemize}
    \item $m_{s, s^{\prime}}=2$ if and only if $s$ and $s^{\prime}$ are commute;
    \item $m_{s,s^{\prime}}=\infty$ if and only if the element $s\cdot s^{\prime}$ has an infinite order.
\end{itemize}
\end{remark}
 


\begin{definition}\label{cox-W-S}
The pair $(W, S)$ is called a Coxeter system, if $W$ is a Coxeter group and $S$ is  set of Coxeter generators. The cardinality of $S$ is called the rank of $(W, S)$. 
 \end{definition}

\subsection{Standard geometric representation} 

Now, we recall the definition of the standard geometric representation of Coxeter groups, as it is defined in \cite{BOOK}.

\begin{definition}
Reflection systems are defined by a triple $(W, S, B)$, where $|S|=n$, and $B$ is the symmetric bilinear form on $\mathbb{R}^{n}$ with the standard basis $\{e_1, \ldots, e_n\}$, which is uniquely determined by the Coxeter system $(W, S)$ as follow:
\begin{itemize}
    \item $B\left(e_{i}, e_{i}\right) = 2$;
    \item $B\left(e_{i}, e_{j}\right)=-2 \cos \left(\frac{\pi}{m_{i, j}}\right)$, ~where, $m_{i,~j}$ are given by the Coxeter system $(W,S)$.
\end{itemize}

\end{definition}



Now, we recall the definition of degenerate and non-degenerate bilinear forms.

\begin{definition}\label{non-degenerate}
Let $B$ be a bilinear form on $\mathbb{R}^n$ with the standard basis $\{e_1, \ldots, e_n\}$. 
\begin{itemize}
    \item If the corresponding matrix $B$ (where $(B)_{i,~j}:=B(e_i, e_j)$~) is singular, then the bilinear form is called degenerate blinear form;
    \item Otherwise, if the corresponding matrix $B$ is regular (non-singular), then the bilinear form is called non-degenerate blinear form.
\end{itemize}
\end{definition}

\begin{definition}\label{s_i-e_i}
The standard geometric representation is the faithful homomorphism $$\pi: W\rightarrow GL_n(\mathbb{R}),$$
where $GL_n(\mathbb{R})$ is  the group of $n\times n$ invertible matrices over $\mathbb{R}$,  and for every $w\in W$,  ~$\pi_w$ is uniquely determined by the homomorphic images of $\pi_{s_i}$ for $s_i\in S$ ($1\leq i\leq n$),  as follow: For each generator  $s_{i} \in S$, ~$\pi_{s_i}$ is defined to be the following reflection:
$$
\pi_{s_{i}}\left(e_{j}\right)=e_{j}-\frac{2 B\left(e_{i}, e_{j}\right)}{B\left(e_{i}, e_{i}\right)} e_{i}.
$$
\end{definition}

One will note that $\pi_{s_{i}}$ negates the corresponding basis vector $e_{i}$ and fixes the set of vectors which are orthogonal to $e_{i}$ relative to the bilinear form $B$. The reflection group $W$ is generated by the set $S=\left\{s_{1}, s_{2}, \ldots, s_{n}\right\}$ and has an action on $\mathbb{R}^n$ which preserves the bilinear form $B$.

\begin{remark}
Every Coxeter group is uniquely determined by a standard geometric representation.
\end{remark}

\subsection{The background and the motivation of the paper}

There are a several generalizations of the standard geometric representation of Coxeter groups, and for various motivations. The most natural generalization have been done by Cameron, Seidel and Tsaranov, where they introduced in 1994 the idea of generalizing the standard geometric representation of Coxeter groups to representation induced from signed graphs, i.e., the edges are signed by 1 or by -1 \cite{CST}. The group they get from the signed geometric representation is a generalized Coxeter group, which is a certain quotient of the Coxeter group which we get from the standard geometric representation of the corresponding non-signed graph, where the non-Coxeter relations in the generalized Coxeter group are induced from cycles, with odd number of negative signs.  There is a series of five papers \cite{rtv, ast1, alv, ast2, line-graph}, which are classifying generalized Coxeter and generalized Artin groups, where looking at the dual graph (a graph which we get by exchanging the vertices by the edges of a given graph), and at the dual signed-graph to a Coxeter line-graph (a graph where there is a possibility to exchange the vertices by the edges), with the Coxeter generators are labeled by the edges, motivated by classifying fundamental groups of branch curves of algebraic surfaces. The idea of \cite{CST} has been generalized to weighted graphs, where it differs from signed graph by the allowance of labeling the edges by any number (not necessarily 1 or -1) \cite{NOT}. The corresponding weighted Coxeter group is a quotient of a Coxeter group where the non-Coxeter relations are induced from cycles with product of the weights along the cycle equals to a root of unity (i.e., an element of finite order in the multiplicative group $\mathbb{C}^{*}$).  
\\

In the paper we deal with a different generalization of Coxeter graphs, namely, simply-laced mixed-sign Coxeter graphs, i.e., the vertices of the graph are signed  by $1$ or by $-1$. There is an associated representation, which is a generalization of the standard geometric representation of Coxeter groups, which was introduced by Hironaka in 2011 \cite{HIR11}. The group which we get by the generalization of Hironaka is called mixed-sign Coxeter group.  The motivation of Hironaka to define mixed-sign Coxeter graphs and mixed-sign Coxeter groups comes from studying the construction of Pseudo-Anosov mapping classes from generalized Coxeter graphs (see \cite{HIR11}). In 2012, Armstrong has  showed in his Ph. D. thesis \cite{Japan Thesis},  that every mixed-sign Coxeter group is a quotient of a certain Coxeter group, whose graph depends on the signs of the vertices of the corresponding  mixed-sign Coxeter graph, without explaining the relations.
In the paper we study the structure of mixed-sign Coxeter group, where we give description of it in terms of generators and relations 
. We classify simply-laced mixed-sign Coxeter groups in terms of generators and relations, which Coxeter graph is either a line or a simple cycle.
The quotients of Coxeter groups, with the relations induced by a mixed-sign Coxeter graph which is a simply-laced simple cycles have importance in other aspects of mathematics as well, like classifying seeds in a Cluster Algebra \cite{BM}.  The idea of labeling edges or vertices of a given graph by a given group elements, and considering relations concerning the cycles of the given graph has been used in completely different terms as well. For example, the idea has been used in  the aspect of small cancellation theory to prove where the Freiheitsatz holds for one-relator free products \cite{S1, S2, HS}. Other direction is  enumerations of the possibly labeling of a given undirected or directed graph by elements of a given group, where the relations of the given group are connected to the labels on the cycle of the graph \cite{CGLS}.

\subsection{Mixed-sign geometric representation}\label{mix-sign-g-r}

We start with recalling the definition of mixed-sign Coxeter groups and mixed-sign geometric representation as it is defined by Hironaka and Armstrong in \cite{Japan Thesis, HIR11}. Then we give some important propositions concerning properties  of the mixed-sign geometric representation which we use in the proofs of the main theorems.

\begin{definition}\label{mixed-geo-rep}
A mixed-sign Coxeter graph $\Gamma$  is defined as an undirected, vertex-labeled graph with $n$ vertices denoted by $s_i$ for $1\leq i\leq n$, and signs of $"+1"$ or $"-1"$ on the vertices  of $\Gamma$. The sign of the vertex $s_i$ is denoted by $f_i$. As we have already seen, we use the notation $m_{i, j}$ for the edge weight between vertices $s_{i}$ and $s_{j}$, where: 
\begin{itemize}
    \item $m_{i, j}=2$ if $s_{i}$ and $s_{j}$ are not connected by an edge;
    \item $m_{i,~j}=3$ if $s_{i}$ and $s_{j}$ are connected by an unlabeled edge (which is called simply-laced edge);
    \item $m_{i,~j}\geq 4$ if $s_{i}$ $s_{j}$ are connected by an edge labeled by $m_{i,~j}$ (which is called non-simply-laced edge)
\end{itemize}

Define a symmetric bilinear form on $\mathbb{R}^{n}$ asociated to the mixed-sign Coxeter graph $\Gamma$, with $n$ vertices $s_1, s_2, \ldots, s_n$ as follow: \\
$$
B\left(e_{i}, e_{j}\right)=\left\{\begin{array}{cc}
2 & \text { if } i=j \text { and vertex } s_{i} \text { has label }+1 \\ \\
-2 & \text { if } i=j \text { and vertex } s_{i} \text { has label }-1 \\ \\
-2 \cos \left(\frac{\pi}{m_{i, j}}\right) & \text { if } i \neq j \\ \\
-2 & \text { if } m_{i, j}=\infty
\end{array}\right.
$$ \\
where for $1\leq i\leq n$,  ~$e_{i}$ is the $i$-th standard basis vector for $\mathbb{R}^{n}$. 

Let $S=\{s_1, s_2, \ldots s_n\}$. Consider the following map $$\pi: S\rightarrow GL_{n}(\mathbb{R})$$ where, 
$$\pi_{s_{i}}\left(e_{j}\right)=e_{j}-\frac{2 B\left(e_{i}, e_{j}\right)}{B\left(e_{i}, e_{i}\right)} e_{i}.$$

Now, we define a group $W(\Gamma)$, which is generated by $S=\{s_1, s_2, \ldots, s_n\}$, and extends the map $\pi: S\rightarrow GL_{n}(R)$ to a group homomorphism $\pi: W(\Gamma)\rightarrow GL_n(R)$. The representation $$\pi: W(\Gamma)\rightarrow GL_n(R)$$ is called the mixed-sign geometric representation of $W(\Gamma)$.

By the definition of $\pi_{s_i}$ for $1\leq i\leq n$, it is easy to see that $s_i$ is an involution in $W(\Gamma)$.
The group $W(\Gamma)$,  which is generated by the involutions $s_i$ for $1\leq i\leq n$, is called mixed-sign Coxeter group, and the pair $(W(\Gamma),S)$, where $W(\Gamma)$ is the mixed-sign Coxeter group associated to the mixed-sign Coxeter graph $\Gamma$ and  which is generated by the set of involutions $S=\{s_1, s_2, \ldots, s_n\}$ is called mixed-sign Coxeter system. 
Notice, in particular, every classical Coxeter system is a mixed-sign Coxeter system.\\
\end{definition}

\begin{remark}
 In our paper we consider simply-laced mixed-sign Coxeter groups only. Hence, for a mixed-sign graph $\Gamma$, the the associated mixed-sign geometric representation has the following form:
 
  $$\pi: S\rightarrow GL_{n}(\mathbb{R})$$ where, 
$$\pi_{s_{i}}\left(e_{j}\right)=\left\{\begin{array}{ll} 
-e_j & \quad{\text{if} ~~i=j} \\ \\
 e_j  & \quad{\text{if} ~~m_{i,~j}=2} \\ \\
e_{j}+f_{i}e_{i}
      & \quad{\text{if} ~~m_{i,~j}=3}
  \end{array}\right.$$
\end{remark}

\subsection{Basic properties of mixed-sign Coxeter groups}

In this subsection we give some basic properties of mixed-sign Coxeter groups, concerning the associated mixed-sign Coxeter graph,  which generalizes  properties of Coxeter groups.

\begin{definition}
Let $W$ be a group generated by $s_1, s_2, \ldots, s_n$ such that $s_i^2=1$. 
\begin{itemize}
    \item Then a relation of a form $\left(s_i\hat{s}_j\right)^m=1$, where $\hat{s}_j$ is an arbitrary conjugate of $s_j$, is called a generalized Coxeter relation;
    \item If all the relations in $W$ are generalized Coxeter relations, then $W$ is called generalized Coxeter group;
    \item For $1\leq k\leq n$, let  $\hat{s}_k$, $\check{s}_k$ be two  arbitrary conjugates of $s_k$ , then for every $1\leq i,~j\leq n$, any relation of a form $(\hat{s}_i\check{s}_j)^m=1$ can be rewritten as a  relation of a form $({s}_i\breve{s}_j)^m=1,$ where $\breve{s}_j$ is a specific conjugate of $s_j$.
\end{itemize}
\end{definition}




Now, we mention some basic properties of mixed-sign Coxeter groups, which is widely used in the paper.

\begin{proposition}\label{2-+}
Consider a mixed-sign Coxeter group generated by $s_1, s_2, \ldots, s_n$, where the corresponding vertices in the associated mixed-sign Coxeter graph are signed by $f_1, f_2, \ldots f_n$ respectively, and for $1\leq i,~j\leq n$, let $m_{i,~j}$ as it is defined in Definition \ref{mixed-geo-rep}.
Then by \cite{Japan Thesis},  the order of the element $s_i s_j$, is as follow:
\begin{itemize}
    \item In case $s_i$ and $s_j$ are connected by an edge (i.e., $m_{i,~j}\geq 3$):
    \begin{itemize}
        \item If $f_i=f_j$ then $|s_i s_j|=m_{i,~j}$;
        \item If $f_i\neq f_j$ then $|s_i s_j|=\infty$.
    \end{itemize}
    \item In case $s_i$ and $s_j$ are not connected by any edge (i.e., $m_{i,~j}=2$), $|s_i s_j|=2$ without any dependence on the sign of the vertices $s_i$ and $s_j$.
\end{itemize}

\end{proposition}

\begin{example}
 consider the following mixed-sign Coxeter graph:

\begin{center}
	\includegraphics[width=.40\textwidth]{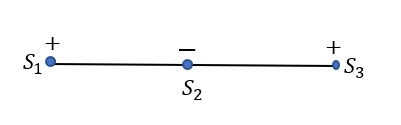}
	\end{center}

By considering the relations in the graph:  $|s_1s_2|=\infty$ and $|s_2s_3|=\infty$.  We will show by Theorem \ref{line-3-relation},
there is one more relation: $(s_1[s_2s_3s_2])^3=1$. Therefore, the mixed-sign Coxeter group, is not a Coxeter group, just  a quotient of a Coxeter group. 
The presentation of the mixed-sign Coxeter group in terms of generators and relations is as follow:
\begin{align} \label{0}
\langle s_{1}, s_{2},s_{3} |s_{1}^{2}=s_{2}^{2}=s_{3}^{2}=1,\left(s_{1} s_{2} s_{3} s_{2}\right)^{3}=1  \rangle.
\end{align} \
\end{example}

\begin{remark}\label{generalized-mixed-sign-coxeter}
Let $\Gamma$ be a mixed-sign Coxeter graph (i.e., every vertex of $\Gamma$ is labeled either by +1 or by -1). 
\begin{itemize}
    \item Then the group $W(\Gamma)$ is not necessarily a Coxeter group. 
    \item There exists non-isomorphic mixed-sign Coxeter graphs $\Gamma$ and $\Gamma^{\prime}$ such that $W(\Gamma)\approx W(\Gamma^{\prime})$.
\end{itemize}

\end{remark}

\begin{definition} \label{bipartite}
A graph $G$ is called  bipartite if the vertices can be divided into two disjoint
sets $N_1$ and $N_2$ such that every edge in $G$ connects a vertex in $N_1$ to a vertex in $N_2$.
\end{definition}

\begin{remark}
Definition \ref{bipartite} implies that any graph that is a line or a simple cycle  of an even length (i.e., simple cycle with $2n$ vertices) is a bipartite graph.
\end{remark}

\begin{definition}\label{g=g_-}
Let $\Gamma$ be a mixed-sign Coxeter graph. Then  $-\Gamma$ is the mixed-sign Coxeter  graph with the same vertices and edges as of $\Gamma$, where every vertex in $-\Gamma$ is labeled  differently to the label of the same vertex in $\Gamma$.
\end{definition}
\begin{remark}\label{MC+=MC-}
Let $\Gamma$ be a mixed-sign Coxeter bipartite graph. Then 
by \cite{Japan Thesis}, $W(\Gamma) \approx W(-\Gamma)$. 
\end{remark}

\subsection{Work plan}

Our work plan in this study is as follow:
\medskip

In Section \ref{Introduction}  we introduce the basic concepts and definitions concerning mixed-sign Coxeter groups. In Section \ref{Pre} we give some important lemmas and propositions which we use in the proofs of the main theorems. Then  in Section \ref{simply-laced-line-cycle} we  study mixed-sign Coxeter groups such that the associated graph is either a simply-laced line or a simply-laced simple cycle.   Finally, in Section \ref{future plans} we  present conclusions and ideas for future plans for further research.\\

\subsection{The main results of the paper}

\begin{itemize}
    \item \textbf{Theorem \ref{line-3-relation}:}
    Let ~$\Gamma$ be a simply-laced mixed-sign Coxeter graph which is a line with $n$ vertices $s_1, s_2, \ldots, s_n$ which are signed by $f_1, f_2, \ldots, f_n$ respectively, such that for $1\leq i\leq n-1$ the vertex $s_i$ is connected to $s_{i+1}$ by a simply-laced edge, as it is described in Fig 1. Then, for $1\leq i<j\leq n$, apart from the Coxeter relations:

\begin{align}
 (s_i\cdot s_j)^2=1 \ \Leftrightarrow j-i\geq 2.   
\end{align}

\begin{align}
(s_i\cdot s_{i+1})^3=1 \ \Leftrightarrow f_i=f_{i+1}.    
\end{align}

which hold by Proposition \ref{2-+}, the following generalized Coxeter relations hold in $W(\Gamma)$:
\begin{align} 
(s_i\cdot s_{i+1, j, i+1})^3=1\quad \text{if and only if}\quad f_i=f_j \end{align}
    
    \item \textbf{Theorem \ref{mixed-sign-s-laced-cycle}:}
    Let $\Gamma$ be a simply-laced simple cycle with vertices $s_1, s_2, \ldots, s_n$  such that the vertex $s_i$ is signed by $f_i$ for  $1\leq i\leq n$, and the vertices $s_i$ and $s_{i+1}$ are connected by a simply-laced edge for $1\leq i\leq n$, where by Remark \ref{n+1=1}, we consider $n+1$ to be $1$ and $1-1$ to be $n$. Then for $1\leq i,~j\leq n$, apart from the Coxeter relations:

\begin{align}
 (s_i\cdot s_j)^2=1 \ \Leftrightarrow 2\leq |i-j|\leq n-2.   
\end{align}

\begin{align}
(s_i\cdot s_{i+1})^3=1 \ \Leftrightarrow f_i=f_{i+1}.    
\end{align}

which hold by Proposition \ref{2-+}, the following generalized Coxeter relations hold in $W(\Gamma)$:

\begin{align} 
(s_{i}\cdot s_{i+1,~j,~i+1})^3=1 \ \Leftrightarrow \ j\neq i-1\quad \text{and}\quad f_i=f_j 
\end{align}

\begin{align}
(s_{j+1, i, j+1}\cdot s_{i+1, j, i+1})^2=1 \ \Leftrightarrow \ \frac{\prod_{k=1}^n f_k}{f_i\cdot f_j} =-1
.\end{align}
     
\end{itemize}

\section{Preliminary results}\label{Pre}

In this section we give some important properties of the mixed-sign geometric representation.

\subsection{Definitions and notations}

Now, we define some definitions and notations concerning the mixed-sign geometric representation.

\begin{definition}\label{E-i-j}
For every $1\leq i,~j\leq n$ such that $i\neq j$ let define the matrix $E_{i,~j}$ of $M_n(\mathbb{R})$ which exchanges the $i$-th row of the identity matrix $I_n$ by the $j$-th row. The entries of $E_{i,~j}$ are as follow:
\begin{itemize}
\item $(E_{i,~j})_{i,~i}=(E_{i,~j})_{j,~j}=0$;
\item $(E_{i,~j})_{i,~j}=(E_{i,~j})_{j,~i}=1$;
\item $(E_{i,~j})_{k,~k}=1$, ~for $1\leq k\leq n$ such that $k\neq i$ and $k\neq j$;
\item $(E_{i,~j})_{k,~\ell}=0$, ~for $k\neq \ell$ ~ and $(k,~\ell)\notin \{(i,~j), ~(j,~i)\}$.
\end{itemize}
\end{definition}

\begin{remark}\label{E-i-j-exchange}
Let $E_{i,~j}$ be a matrix of $M_n(\mathbb{R})$ as it is defined in Definition \ref{E-i-j}, then for every matrix $A\in M_n(\mathbb{R})$, the matrix which we get by multiplying $A$ by $E_{i,~j}$ at the left (i.e., $E_{i,~j}\cdot A$) is the matrix which exchanges the $i$-th row of $A$ by the $j$-th row.
\end{remark}

\begin{definition}\label{tau}
Consider the mixed-sign geometric representation of a simply-laced mixed-sign Coxeter group $W$, as it is defined in Definition \ref{mixed-geo-rep}. Then, for $1\leq i\leq n$,  
we denote by $\tau_i$ the matrix in $M_n(\mathbb{R})$ (the $n\times n$ matrices over $\mathbb{R}$), which satisfies
$$\tau_{i}\left(e_{j}\right)=-\frac{2 B\left(e_{i}, e_{j}\right)}{B\left(e_{i}, e_{i}\right)} e_{i} =\left\{\begin{array}{ll}  
   -2e_i & {\text{if} ~~i=j} \\ \\
   0  & {\text{if} ~~m_{i,~j}=2} \\ \\
   f_{i}e_{i}
     & {\text{if} ~~m_{i,~j}=3}
\end{array}\right.$$
It is easy to see that 
$$\tau_i=\pi_{s_i}-1\quad \text{where} ~1~ \text{is the identity matrix of} ~M_n(\mathbb{R}).$$
Notice, for $n>1$,  ~$\tau_i$ is not an invertable matrix in $M_n(\mathbb{R})$ since all the rows apart from the $i$-th row of $\tau_i$ are $0$. Therefore, $\tau_i$ is not a representative image of any element of $W$ by the mixed-sign geometric representation. Although $\tau_i$ is widely used in the proofs of the theorems concerning the structure of mixed-sign Coxeter groups. Hence, we present now some important properties of $\tau_i$.
\end{definition}

\subsection{Important properties of $\tau_i$}

This section is about some important properties of products of $\tau_i$, which are used at the proofs of the main results
 of the paper. 
 \\

\begin{proposition}\label{tau-property} For $1\leq i\leq n$, let $\tau_i$ be the element of $M_n(\mathbb{R})$ as it is defined in Definition \ref{tau}, then $\tau_i$ satisfies the following properties:
\begin{enumerate}[label=(\roman*)]
    \item $\tau_i^2=-2\tau_i$
    \item $(1+\tau_i)\tau_i=\tau_i(1+\tau_i)=-\tau_i$
    \item Let $E_{i,~j}$ be a matrix of $M_n(\mathbb{R})$ as it is defined in Definition \ref{E-i-j}. If $m_{i,~j}=3$, then  the following holds:
   \begin{itemize}
       \item  $\tau_i\tau_j=E_{i,~j} \cdot f_i \cdot \tau_{j}$
        \item  $\tau_i\tau_j \tau_i=f_i\cdot f_j\cdot\tau_i$
        \item $\left(\tau_i\tau_j\right)^k=\left(f_i\cdot f_j\right)^{k-1}\cdot \tau_i \tau_j = E_{i,~j} \cdot\left(f_i\right)^k\cdot \left(f_j\right)^{k-1}\cdot \tau_{j}$
        \item $\left(\tau_i\tau_j\right)^k\tau_i=\left(f_i\cdot f_j\right)^{k}\cdot \tau_i$
  \end{itemize}
  \item For $i\neq j$:
  \\
  
  $m_{i,~j}=2\Leftrightarrow (s_is_j)^2=1 \ \Leftrightarrow \ s_is_j=s_js_i \ \Leftrightarrow \  \tau_i\tau_j=\tau_j\tau_i=0$.

\end{enumerate}
\end{proposition}
\medskip





\bigskip

\begin{proposition}\label{tau-i1-ik}
 Consider a simply-laced mixed-sign Coxeter group $W$, with generators $s_1, s_2, \ldots, s_n$. For $1\leq i,~j\leq n$, let $\tau_i$ and  ~$E_{i,~j}$ be as they  are defined in Definitions \ref{tau} and  \ref{E-i-j} respectively. Consider the product $ \tau_{i_1}\cdot \tau_{i_2}\cdots \tau_{i_k}$, where denote by $\rho$ the number of $1\leq t\leq k-1$ such that $\tau_{i_t}=\tau_{i_{t+1}}$. Then by Proposition \ref{tau-property}, for every $1\leq i_1, \ldots, i_k\leq n$, the following holds:
 
    $$ \tau_{i_1}\cdot \tau_{i_2}\cdots \tau_{i_k} = \left\{\begin{array}{ll}(-2)^{\rho}\cdot \prod_{t=1 ~| ~\tau_{i_{t}}\neq \tau_{i_{t+1}}}^{k-1} f_{i_t}\cdot E_{i_1, i_k}\cdot \tau_{i_k} & \quad\quad \text{If} ~~ m_{i_{t}, i_{t+1}}=3~~~\text{for all}~~1\leq t\leq k-1 \\ \\ 0 & \quad\quad \text{If} ~~ m_{i_{t}, i_{t+1}}=2~~~\text{for some}~~1\leq t\leq k-1 \end{array}\right.$$

\end{proposition}

\begin{proof}
The proof comes by applying $k$ times Proposition \ref{tau-property}. 
\end{proof}

\begin{cor}\label{tau-first-non-zero}
Consider a simply-laced mixed-sign Coxeter group $W$, with generators $s_1, s_2, \ldots, s_n$. Then, for $1\leq i_1, i_2, \ldots, i_k\leq n$, where $k$ is a positive integer, 
 every non-zero element of $M_{n}(\mathbb{R})$ of the form $\tau_{i_1}\cdot \tau_{i_2}\cdots \tau_{i_k}$ satisfy the following properties:
 \begin{itemize}
     \item $\tau_{i_1}\cdot \tau_{i_2}\cdots \tau_{i_k}$ contains only one non-zero row;
      \item The non-zero row of $\tau_{i_1}\cdot \tau_{i_2}\cdots \tau_{i_k}$ is the $i_1$-th row;
      \item The $i_1$-th row of $\tau_{i_1}\cdot \tau_{i_2}\cdots \tau_{i_k}$  is a multiple of the non-zero (the $i_k$-th row) row of $\tau_{i_k}$.
 \end{itemize}
\end{cor}

\begin{proposition}\label{sum-tau-i-beginning}
Let $W$ be a simply-laced mixed-sign Coxeter group, which is generated by $s_1, s_2, \ldots, s_n$. Let $t$ be a non-zero element of $\mathbb{R}$ and $1\leq i_1, i_2, \ldots, i_k\leq n$.
Assume  sum of elements of the form $t\cdot \tau_{i_1}\tau_{i_2}\cdots \tau_{i_k}$ equals to $0$. Then, sum of the elements with $\tau_p$ at left equals to $0$ for every $1\leq p\leq n$.
\end{proposition}

\begin{proof}
Assume 
\begin{equation}\label{tau-product}
    t_1\cdot \tau_{i_{1,1}}\tau_{i_{1,2}}\cdots \tau_{i_{1, k_1}} + t_2\cdot \tau_{i_{2,1}}\tau_{i_{2,2}}\cdots \tau_{i_{2, k_2}} + \cdots + t_{\ell}\cdot \tau_{i_{\ell,1}}\tau_{i_{\ell,2}}\cdots \tau_{i_{\ell, k_{\ell}}}=0.
\end{equation}

where:
\begin{itemize}
    \item $\ell$ and $k_j$ for $1\leq j\leq \ell$ are positive integers;
    \item $1\leq i_{u, v}\leq n$ for $1\leq u\leq \ell$ and $1\leq v\leq k_u$;
    \item  $t_1, t_2, \ldots, t_{\ell}$ are non-zero elements of $\mathbb{R}$.
\end{itemize}
  By Corollary  \ref{tau-first-non-zero}, for every $1\leq q\leq \ell$,  a summand of a form $t_q\cdot \tau_{i_{q,1}}\tau_{i_{q,2}}\cdots \tau_{i_{q, k_1}}$, with $i_{q,1}=p$ for some $1\leq p\leq n$ is a matrix in $M_n(\mathbb{R})$ with one non-zero row at the $p$-th row. Since sum of all the summands at the left hand side of Equation \eqref{tau-product} equals to $0$, sum of all the summands  at the left hand side of Equation \eqref{tau-product} with $\tau_p$ at left (which are the summands with non-zero row only at the $p$-th row)  equals to $0$, for every $1\leq p\leq n$.
\end{proof}

\begin{example}
Consider the Coxeter group 
$$\widetilde{S}_3=\langle s_1, s_2, s_3 ~|~ s_1^2=s_2^2=s_3^2=1, ~(s_1s_2)^3=(s_2s_3)^3=(s_3s_1)^3=1\rangle.$$
Then, 
$$\tau_1 + \tau_1\tau_2 +\tau_1\tau_3 + \tau_3 + \tau_3\tau_2+\tau_3\tau_1+\tau_2+\tau_2\tau_1+\tau_2\tau_3=0.$$
Notice, 
$$\tau_1 + \tau_1\tau_2 +\tau_1\tau_3=0,\quad  \tau_2+\tau_2\tau_1+\tau_2\tau_3=0,\quad  \tau_3 + \tau_3\tau_2+\tau_3\tau_1=0,$$
where $f_1=f_2=f_3=1$ and:

$$
\tau_1=\left(\begin{array}{ccc}
-2 & 1 & 1 \\
0 & 0 & 0 \\
0 & 0 & 0
\end{array}\right) \quad \tau_2=\left(\begin{array}{ccc}
0 & 0 & 0 \\
1 & -2 & 1 \\
0 & 0 & 0
\end{array}\right) \quad \tau_3=\left(\begin{array}{ccc}
0 & 0 & 0 \\
0 & 0 & 0 \\
1 & 1 & -2
\end{array}\right)
$$

\end{example}

\begin{definition}\label{nu}
For every $1\leq i\leq n$, the vector  $\nu_{i}\in \mathbb{R}^n$ is defined to be the $i$-th row of $\tau_i$.
\end{definition}

\begin{proposition}\label{sum-tau-i-finishing}
Let $W$ be a simply-laced mixed-sign Coxeter group, which is generated by $s_1, s_2, \ldots, s_n$. Let $0\neq t\in \mathbb{R}$ and $1\leq i_1, i_2, \ldots, i_k\leq n$ for some positive integer $k$.
Assume  the following holds:
\begin{itemize}
    \item Sum of elements of the form $t\cdot \tau_{i_1}\tau_{i_2}\cdots \tau_{i_k}$ equals to $0$;
    \item The associated bilinear form is a non-degenerate bilinear form (i.e., the set of vectors \\ $\{\nu_{1}, \nu_{2}, \ldots, \nu_{n}\}$ are linearly independent in $\mathbb{R}^n$).
\end{itemize}
 Then, sum of the elements with $\tau_p$ at left and $\tau_q$ at right, equals to $0$ for every ordered pair $(p,q)$ such that $1\leq p, q\leq n$.
\end{proposition}

\begin{proof}

Consider Equation \eqref{tau-product}, with all the notations. By Proposition \ref{sum-tau-i-beginning}, for every $1\leq p\leq n$ the sum of all the summands at the left hand side of Equation \eqref{tau-product} with $\tau_p$ at left equals to $0$.  
By Corollary \ref{tau-first-non-zero}, every summand at the left hand side of of Equation \eqref{tau-product} with $\tau_p$ at left and $\tau_{q}$ at right is a matrix in $M_n(\mathbb{R})$ with the following properties:
\begin{itemize}
    \item The only non-zero row is the $p$-th row;
    \item The $p$-th row is a multiply of $\nu_{q}$ (where $\nu_q$ is defined in Definition \ref{nu}).
\end{itemize}
Hence, the sum of all the summands at the left hand side of Equation \eqref{tau-product} with $\tau_p$ at left is a matrix with a linear combination of the vectors  $\{\nu_{1}, \nu_{2}, \ldots, \nu_{n}\}$ at the $p$-th row. Since the set of vectors $\{\nu_{1}, \nu_{2}, \ldots, \nu_{n}\}$ are linearly independent, it is satisfied that for every $1\leq q\leq n$ sum of all the summands at the left hand side of Equation \eqref{tau-product} with $\tau_p$ at left and $\tau_q$ at right equals to $0$. 

\end{proof}

Now, we recall the definition of Dickson polynomials  as it is defined by Dickson  \cite{D}. 
\\

The Dickson polynomials of the second kind $\mathbb{E}_{n}(x, \alpha)$ are defined by the following recurrence relation for $n \geq 2$:
\\

$$
\mathbb{E}_{n}(x, \alpha)=x \mathbb{E}_{n-1}(x, \alpha)-\alpha \mathbb{E}_{n-2}(x, \alpha),
$$
with the initial conditions $\mathbb{E}_{0}(x, \alpha)=1$ and $\mathbb{E}_{1}(x, \alpha)=x$.
In the following claim we use Dickson polynomial $\mathbb{E}_{n}{\left(x, \alpha \right)}$ where $x=1$ and $\alpha=f_{i} f_{j}$ (reminder: $f_{i} f_{j}$ is either $+1$ or $-1$).
\\
and get $\mathbb{E}_{n}{\left(1, f_{i} f_{j}\right)}$ with the same markings we set up on top. \\

\begin{proposition}\label{r-dickson}
Let $\tau_i$ and $\tau_j$ be as they are defined in Definition \ref{tau}   and let $r$ be a positive integer then the following holds: 

\begin{enumerate}
 \item 
\begin{align} \label{ase}
{\left[\left(1+\tau_{i}\right)\left(1+\tau_{j}\right)\right]^{r}-\left[\left(1+\tau_{j}\right)\left(1+\tau_{i}\right)\right]^{r}=}\left(f_{i} f_{j}\right)^{r-1} \cdot \left\{\mathbb{E}_{2 r-1}{\left(1~,~ f_{i} f_{j}\right)}\right\} \cdot E_{i, j} \cdot\left(f_{i}  \tau_{j}-f_{j}  \tau_{i}\right)
\end{align}
\item 
\begin{align} \label{aso}
[(1+\tau_{i})(1+\tau_{j})]^{r}(1+\tau_{i}) 
-[(1+\tau_{j})(1+\tau_{i})]^{r}(1+\tau_{j})
=\left(f_{i} f_{j}\right)^r \cdot \left\{\mathbb{E}_{2 r}{\left(1~, ~f_{i} f_{j}\right)}\right\} \cdot\left(\tau_{i}-\tau_{j}\right)
\end{align}
\end{enumerate}
\end{proposition}

\begin{proof}
We prove the proposition by induction on $r$:
\\

First, consider  $r=1$ in Equation \eqref{ase}.

$$
\left(1+\tau_{i}\right)\left(1+\tau_{j}\right)-\left(1+\tau_{j}\right)\left(1+\tau_{i}\right)=\tau_{i} \tau_{j}-\tau_{j} \tau_{i}=1 \cdot E_{i, j} \cdot\left(f_{i}  \tau_{j}-f_{j}  \tau_{i}\right)$$
$$
=\left(f_{i} f_{j}\right)^{0} \cdot \left\{\mathbb{E}_{1}{(1~,~ f_{i} f_{j})}\right\} \cdot
E_{i, j}\left(f_{i} \cdot \tau_{j}-f_{j} \cdot \tau_{i}\right).
$$

Now, consider $r=1$ in Equation \eqref{aso}
$$
\left(1+\tau_{i}\right)\left(1+\tau_{j}\right)\left(1+\tau_{i}\right)-\left(1+\tau_{j}\right)\left(1+\tau_{i}\right)\left(1+\tau_{j}\right)=
\left(\tau_{i} \tau_{j} \tau_{i} - \tau_{i}\right)-\left(\tau_{j} \tau_{i}\tau_{j}-\tau_{j}\right)
$$
$$
=\left(f_{i} f_{j} -1\right)\left(\tau_{i}-\tau_{j}\right)
=\left(f_{i} f_{j}\right)^{1} \cdot \left\{\mathbb{E}_{2}{\left(1~, ~f_{i} f_{j}\right)}\right\}\cdot \left(\tau_{i}-\tau_{j}\right).
$$

Assume Equations \eqref{ase} and \eqref{aso} holds for $r=r_0$:

\begin{align} \label{pre}
{\left[\left(1+\tau_{i}\right)\left(1+\tau_{j}\right)\right]^{r_0}-\left[\left(1+\tau_{j}\right)\left(1+\tau_{i}\right)\right]^{r_0}=}\left(f_{i} f_{j}\right)^{r_0-1} \cdot \left\{\mathbb{E}_{2 r_0-1}{\left(1~, ~f_{i} f_{j}\right)}\right\} \cdot E_{i, j} \cdot\left(f_{i} \cdot \tau_{j}-f_{j} \cdot \tau_{i}\right)
\end{align}
and 
\begin{align} \label{pro}
[(1+\tau_{i})(1+\tau_{j})]^{r_0}(1+\tau_{i}) 
-[(1+\tau_{j})(1+\tau_{i})]^{r_0}(1+\tau_{j})
=\left(f_{i} f_{j}\right)^{r_0} \cdot \left\{\mathbb{E}_{2 r_0}{\left(1~,~ f_{i} f_{j}\right)}\right\} \cdot\left(\tau_{i}-\tau_{j}\right)
\end{align}

Now, we prove the lemma for $r=r_0+1$, where using Equations \eqref{pre} and \eqref{pro}: 

\begin{equation}\label{r0+1-start}
\begin{aligned} 
&\left[\left(1+\tau_{i}\right)\left(1+\tau_{j}\right)\right]^{r_0+1}-\left[\left(1+\tau_{j}\right)\left(1+\tau_{i}\right)\right]^{r_0+1}\\
&={\left[\left(1+\tau_{i}\right)\left(1+\tau_{j}\right)\right]^{r_0}\left(1+\tau_{i}\right)\left(1+\tau_{j}\right)-\left(1+\tau_{j}\right)\left(1+\tau_{i}\right)\left[\left(1+\tau_{j}\right)\left(1+\tau_{i}\right)\right]^{r_0}}\\
&=\left[\left(1+\tau_{i}\right)\left(1+\tau_{j}\right)\right]^{r_0}\left(1+\tau_{i}\right)
+\left[\left(1+\tau_{i}\right)\left(1+\tau_{j}\right)\right]^{r_0}\left(1+\tau_{i}\right) \tau_{j}\\
&-\left(1+\tau_{i}\right)\left[\left(1+\tau_{j}\right)\left(1+\tau_{i}\right)\right]^{r_0}
-\tau_{j}\left(1+\tau_{i}\right)\left[\left(1+\tau_{j}\right)\left(1+\tau_{i}\right)\right]^{r_0} .
\end{aligned}
\end{equation}
\\
Since $\left[\left(1+\tau_{i}\right)\left(1+\tau_{j}\right)\right]^{r_0}\left(1+\tau_{i}\right)-\left(1+\tau_{i}\right)\left[\left(1+\tau_{j}\right)\left(1+\tau_{i}\right)\right]^{r_0}=0$,

Equation \eqref{r0+1-start} is equivalent to:

\begin{equation}
\begin{aligned} 
&\left[\left(1+\tau_{i}\right)\left(1+\tau_{j}\right)\right]^{r_0}\left(1+\tau_{i}\right) \tau_{j} -\tau_{j}\left(1+\tau_{i}\right)\left[\left(1+\tau_{j}\right)\left(1+\tau_{i}\right)\right]^{r_0}=\\
&\left[\left(1+\tau_{i}\right)\left(1+\tau_{j}\right)\right]^{r_0} \tau_{j}
+\left[\left(1+\tau_{i}\right)\left(1+\tau_{j}\right)\right]^{r_0} \tau_{i}\tau_{j}
-\tau_{j}\left[\left(1+\tau_{j}\right)\left(1+\tau_{i}\right)\right]^{r_0}
-\tau_{j}\tau_{i}\left[\left(1+\tau_{j}\right)\left(1+\tau_{i}\right)\right]^{r_0}
\end{aligned}
\end{equation}

Now, by using $(1+\tau_j)\tau_j=\tau_j(1+\tau_j)=-\tau_j$, as it is described in part (ii) of Proposition \ref{tau-property}  we get the following equation:
\begin{align} \label{r1}
\left[\left(1+\tau_{i}\right)\left(1+\tau_{j}\right)\right]^{r_0} \tau_{j}-\tau_{j}\left[\left(1+\tau_{j}\right)\left(1+\tau_{i}\right)\right]^{r_0}=
-\left[\left(1+\tau_{i}\right)\left(1+\tau_{j}\right)\right]^{r_0}+\left[\left(1+\tau_{j}\right)\left(1+\tau_{i}\right)\right]^{r_0}
\end{align}
 
Therefore, by using Equation \eqref{pre}:
\begin{equation}\label{r4}
\begin{aligned} 
&\left[\left(1+\tau_{i}\right)\left(1+\tau_{j}\right)\right]^{r_0} \tau_{j}-\tau_{j}\left[\left(1+\tau_{j}\right)\left(1+\tau_{i}\right)\right]^{r_0}\\ &=
-\left(f_{i} f_{j}\right)^{r_0-1} \cdot \left\{\mathbb{E}_{2 r_0-1}{\left(1~,~ f_{i} f_{j}\right)}\right\} \cdot E_{i, j} \cdot\left(f_{i} \tau_{j}-f_{j}  \tau_{i}\right)
\end{aligned}
\end{equation}

Since both $f_i$ and $f_j$ are equal either to  $1$ or to $-1$, it is always satisfied $(f_if_j)^2=1$ so one can multiply Equation \eqref{r4} by $(f_if_j)^2=1$ and obtain
\begin{equation}\label{r4+fifj2}
\begin{aligned}
&\left[\left(1+\tau_{i}\right)\left(1+\tau_{j}\right)\right]^{r_0} \tau_{j}-\tau_{j}\left[\left(1+\tau_{j}\right)\left(1+\tau_{i}\right)\right]^{r_0} \\ &=-\left(f_{i} f_{j}\right)^{r_0+1} \cdot \left\{\mathbb{E}_{2 r_0-1}{\left(1~,~ f_{i} f_{j}\right)}\right\} \cdot E_{i, j} \cdot\left(f_{i} \tau_{j}-f_{j}  \tau_{i}\right).
\end{aligned}
\end{equation}

Now we consider the following summands of Equation \eqref{r0+1-start}

$$\left[\left(1+\tau_{i}\right)\left(1+\tau_{j}\right)\right]^{r_0} \tau_{i}\tau_{j}
-\tau_{j}\tau_{i}\left[\left(1+\tau_{j}\right)\left(1+\tau_{i}\right)\right]^{r_0}.$$
\medskip

Now, by using Equation \eqref{pro}:

$$
\begin{aligned}
&\left(f_{i} f_{j}\right)^{r_0}\cdot  \left\{\mathbb{E}_{2 r_0}{\left(1, f_{i} f_{j}\right)}\right\}\cdot\left(\tau_{i} -\tau_{j}\right) \\&=
{\left[\left(1+\tau_{i}\right)\left(1+\tau_{j}\right)\right]^{r_0}\left(1+\tau_{i}\right)-\left[\left(1+\tau_{j}\right)\left(1+\tau_{i}\right)\right]^{r_0}\left(1+\tau_{j}\right)} \\
&={\left[\left(1+\tau_{i}\right)\left(1+\tau_{j}\right)\right]^{r_0}\left(1+\tau_{i}\right)-\left(1+\tau_{j}\right)\left[\left(1+\tau_{i}\right)\left(1+\tau_{j}\right)\right]^{r_0}} \\
&={\left[\left(1+\tau_{i}\right)\left(1+\tau_{j}\right)\right]^{r_0}+\left[\left(1+\tau_{i}\right)\left(1+\tau_{j}\right)\right]^{r_0} \tau_{i}-\left[\left(1+\tau_{i}\right)\left(1+ \tau_{j}\right)\right]^{r_0}-\tau_{j}\left[\left(1+\tau_{i}\right)\left(1+\tau_{j}\right)\right]^{r_0}} \\
&={\underbrace{\left[\left(1+\tau_{i}\right)\left(1+\tau_{j}\right)\right]^{r_0} \tau_{i}}_{\text {A }}
-\underbrace{\tau_{j}\left[\left(1+\tau_{i}\right)\left(1+\tau_{j}\right)\right]^{r_0}}_{\text {B }}}
.
\end{aligned}
$$

Now, notice that $\left[\left(1+\tau_{i}\right)\left(1+\tau_{j}\right)\right]^{r_0}$ is a sum of elements of the form: $(\tau_i\tau_j)^k$, ~$(\tau_i\tau_j)^k \tau_i$, ~$(\tau_j\tau_i)^k$, ~$(\tau_j\tau_i)^k\tau_j$
   for $0\leq k\leq r_0$. Then, by Proposition \ref{tau-property}, there exists real numbers $\rho_{i,~i}$, $\rho_{i,~j}$, $\rho_{j,~i}$, $\rho_{j,~j}$ such that :
    $$\left[\left(1+\tau_{i}\right)\left(1+\tau_{j}\right)\right]^{r_0}=1+\rho_{i,~i}\tau_i + \rho_{i,~j}\tau_i\tau_j + \rho_{j,~i}\tau_j\tau_i + \rho_{j,~j}\tau_j.$$
    
    By symmetry, it is easy to see that $\rho_{i,~i}=\rho_{j,~j}$.
    Hence, we get:
    \begin{equation*}
    \begin{aligned}
       &\left[\left(1+\tau_{i}\right)\left(1+\tau_{j}\right)\right]^{r_0}\tau_i-\tau_j\left[\left(1+\tau_{i}\right)\left(1+\tau_{j}\right)\right]^{r_0} \\ 
       &=\left(1+\rho_{i,~i}\tau_i + \rho_{i,~j}\tau_i\tau_j + \rho_{j,~i}\tau_j\tau_i + \rho_{i,~i}\tau_j\right)\tau_i-\tau_j\left(1+\rho_{i,~i}\tau_i + \rho_{i,~j}\tau_i\tau_j + \rho_{j,~i}\tau_j\tau_i + \rho_{i,~i}\tau_j\right).
    \end{aligned}
    \end{equation*}
    
Then by using Proposition \ref{tau-property}, we get:
 
 \begin{equation}\label{tau-dickson-connection}
 \begin{aligned} 
& \left(f_{i} f_{j}\right)^{r_0}\cdot  \left\{\mathbb{E}_{2 r_0}{\left(1, f_{i} f_{j}\right)}\right\}\cdot\left(\tau_{i} -\tau_{j}\right) \\
& =\left[\left(1+\tau_{i}\right)\left(1+\tau_{j}\right)\right]^{r_0}\tau_i-\tau_j\left[\left(1+\tau_{i}\right)\left(1+\tau_{j}\right)\right]^{r_0} \\
  &=\left(1+\rho_{i,~i}\tau_i + \rho_{i,~j}\tau_i\tau_j + \rho_{i,~j}\tau_j\tau_i + \rho_{i,~i}\tau_j\right)\tau_i-\tau_j\left(1+\rho_{i,~i}\tau_i + \rho_{i,~j}\tau_i\tau_j + \rho_{i,~j}\tau_j\tau_i + \rho_{i,~i}\tau_j\right) \\
  & =\left(1-2\rho_{i,~i}+\rho_{i,~j}\cdot \tilde{f}_{i,~j}\cdot\tilde{f}_{j,~i}\right)\left(\tau_i-\tau_j\right).
 \end{aligned}
\end{equation}

Now, notice,
\begin{equation*}
\begin{aligned}
    &\left[\left(1+\tau_{i}\right)\left(1+\tau_{j}\right)\right]^{r_0} \tau_{i}\tau_{j}
-\tau_{j}\tau_{i}\left[\left(1+\tau_{j}\right)\left(1+\tau_{i}\right)\right]^{r_0}\\
&=\left(1-2\rho_{i,~i}+\rho_{i,~j}\cdot \tilde{f}_{i,~j}\cdot\tilde{f}_{j,~i}\right)\left(\tau_i\tau_j-\tau_j\tau_i\right).
\end{aligned}
\end{equation*}

Hence, by Equation \eqref{tau-dickson-connection},
\begin{equation}
    \begin{aligned}
    &\left[\left(1+\tau_{i}\right)\left(1+\tau_{j}\right)\right]^{r_0} \tau_{i}\tau_{j}
-\tau_{j}\tau_{i}\left[\left(1+\tau_{j}\right)\left(1+\tau_{i}\right)\right]^{r_0} \\ &=\left(f_{i} f_{j}\right)^{r_0} \left\{\mathbb{E}_{2 r_0}{\left(1~,~ f_{i} f_{j}\right)}\right\}\left(\tau_{i}\tau_{j} -\tau_{j}\tau_{i}\right).
    \end{aligned}
\end{equation}

By using Proposition \ref{tau-property} for the expression $\left(\tau_{i}\tau_{j} -\tau_{j}\tau_{i}\right)$, 
we get
\begin{equation} \label{r5}
\begin{aligned}
&\left[\left(1+\tau_{i}\right)\left(1+\tau_{j}\right)\right]^{r_0} \tau_{i}\tau_{j}
-\tau_{j}\tau_{i}\left[\left(1+\tau_{j}\right)\left(1+\tau_{i}\right)\right]^{r_0}\\
&=(f_if_j)^{r_0}\cdot1\cdot \left\{\mathbb{E}_{2r_0}{\left(1~,~ f_{i} f_{j}\right)}\right\}\cdot E_{i,~j}\cdot(f_i\tau_j-f_j\tau_i).
\end{aligned}
\end{equation}

Finally we have:

$$
\begin{aligned}
&\left[\left(1+\tau_{i}\right)\left(1+\tau_{j}\right)\right]^{r_0+1}-\left[\left(1+\tau_{j}\right)\left(1+\tau_{i}\right)\right]^{r_0+1}\\
&=\left[\left(1+\tau_{i}\right)\left(1+\tau_{j}\right)\right]^{r_0}\left(1+\tau_{i}\right) \tau_{j} -\tau_{j}\left(1+\tau_{i}\right)\left[\left(1+\tau_{j}\right)\left(1+\tau_{i}\right)\right]^{r_0}\\
&=\left[\left(1+\tau_{i}\right)\left(1+\tau_{j}\right)\right]^{r_0} \tau_{j}
+\left[\left(1+\tau_{i}\right)\left(1+\tau_{j}\right)\right]^{r_0} \tau_{i}\tau_{j}
-\tau_{j}\left[\left(1+\tau_{j}\right)\left(1+\tau_{i}\right)\right]^{r_0}
-\tau_{j}\tau_{i}\left[\left(1+\tau_{j}\right)\left(1+\tau_{i}\right)\right]^{r_0}\\
&=(f_if_j)^r\cdot1\cdot \left\{\mathbb{E}_{2r_0}{\left(1~, ~f_{i} f_{j}\right)}\right\}\cdot E_{i,~j}\cdot(f_i\tau_j-f_j\tau_i)\\
&-\left(f_{i} f_{j}\right)^{r_0+1} \cdot \left\{\mathbb{E}_{2 r_0-1}{\left(1~, ~f_{i} f_{j}\right)}\right\} \cdot E_{i, j} \cdot\left(f_{i} \tau_{j}-f_{j}  \tau_{i}\right)\\
&=\left(f_{i} f_{j}\right)^r[1\cdot \left\{\mathbb{E}_{2r_0}{\left(1~, ~f_{i} f_{j}\right)}\right\}-(f_if_j)\cdot\left\{\mathbb{E}_{2r_0-1}{\left(1~, ~f_{i} f_{j}\right)}\right\}]\cdot E_{i,~j}\cdot (f_i\tau_j-f_j\tau_i)\\
&=\left(f_{i} f_{j}\right)^r\cdot \left\{\mathbb{E}_{2r_0+1}{\left(1~, ~f_{i} f_{j}\right)}\right\}\cdot E_{i,~j}\cdot (f_i\tau_j-f_j\tau_i).
\end{aligned}
$$

Hence, Equation \eqref{ase} holds for every $r$.\\

The proof of Equation \eqref{aso} is very similar to the proof of Equation \eqref{ase} and done by using the same technique.
\end{proof}




\section{Mixed-sign Coxeter graphs which are  simply-laced line or simple cycle}\label{simply-laced-line-cycle}

In this section we give a characterization in terms of generators and relations  for the mixed-sign Coxeter groups, which associated graph is a simply-laced line or a simply-laced simple cycle.

\subsection{The mixed-sign Coxeter graph is a simply laced line}\label{simply-laced-line}
 
 Now, we consider mixed-sign Coxeter groups where the associated mixed-sign Coxeter graph is a line with more than two vertices, and all the edges are simply-laced.

\begin{center}
	\includegraphics[width=.43\textwidth]{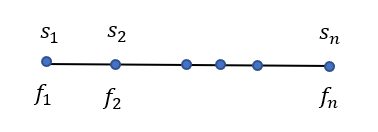}
	\end{center}
	\begin{center}
	
{\scriptsize { \normalsize $Fig. \ 1.$}}
 	    \end{center}
	
 \begin{definition}\label{tau-i,j,i}
 Consider the  mixed-sign Coxeter graph with $n$ vertices $s_1, s_2, \ldots, s_n$ which are signed by $f_1, f_2, \ldots, f_n$ respectively, such that for $1\leq i\leq n-1$ the vertex $s_i$ is connected to $s_{i+1}$ by a simply-laced edge, as it is described in Fig 1. For $i_1, i_2, j$ such that $1\leq i_1\leq j\leq i_2\leq n$, We use the following notations:
 
  \begin{itemize}
  
    \item $s_{i_1, j, i_2} := s_{i_1} s_{i_1+1} \cdots s_{j-1} s_{j} s_{j-1} \cdots s_{i_2+1}s_{i_2}$
    \item $\tau_{\left(i_{1}, j, i_{2}\right)}:=\left(\prod_{k=i_{1}}^{j-1} 
\tau_{k}\right) \tau_{j}\left(\prod_{k=1}^{j-i_{2}} \tau_{j-k}\right)=
\tau_{i_{1}} \ldots \tau_{j-1} \tau_{j}\tau_{j-1} \ldots \tau_{i_{2}}$
  \item $\tau_{i_1 , j , i_2}:=
   \sum_{k_{1}=i_1}^{j} \sum_{k_{2}=i_2}^{j} \tau_{\left(k_{1}, j, k_{2}\right)}$
  \end{itemize}
 \end{definition}

 \begin{proposition}\label{tau-middle}   Consider the mixed-sign Coxeter graph $\Gamma$ and all the notations of Definition \ref{tau-i,j,i}. Then the mixed-sign Coxeter group $W(\Gamma) $ satisfies the following property:
 \begin{equation*}
\begin{aligned}
 \pi_{s_{i_1, j, i_2}}&=(1+\tau_{i_1})(1+\tau_{i_1+1})\cdots(1+\tau_{j-1})(1+\tau_{j})(1+\tau_{j-1})\cdots(1+\tau_{i_2-1})(1+\tau_{i_2}) \\ 
 & = 1+ \sum_{k_{1}=i_1}^{j} \sum_{k_{2}=i_2}^{j} \tau_{\left(k_{1}, j, k_{2}\right)} = 1+\tau_{i_{1}, j, i_{2}},   
\end{aligned}
\end{equation*}

for every $i_1, i_2$, such that $i_1<j$ and $i_2<j$.
\end{proposition}

\begin{proof}

Let define $i:=\max\{i_1, ~i_2\}$.
Now, we prove the proposition by induction on $j-i$. 

For $j-i=1$ the following holds:

 $$\pi_{s_{i} s_{j}  s_{i}}=\left(1+\tau_{i}\right)\left(1+\tau_{j}\right)\left(1+\tau_{i}\right)=1+\tau_{i} \tau_{j}+\tau_{j}\tau_{i}+\tau_{i} \tau_{j} \tau_{i}+\tau_{j}.$$
Hence,

\begin{itemize}
    \item If  $i=i_1$, then  $j=i_1+1$ (since $j-i=1)$ and $i_2<i_1$, and then 
    \begin{equation*}
    \begin{aligned}
        \pi_{s_{i_1, j, i_2}} &=\pi_{s_{i_1} s_j s_{i_1} s_{i_1-1}\cdots s_{i_2}} \\ 
        &= \left(1+\tau_{i}\right)\left(1+\tau_{j}\right)\left(1+\tau_{i}\right)\left(1+\tau_{i_1-1}\right)\cdots \left(1+\tau_{i_2}\right) \\ &=\left(1+\tau_{i} \tau_{j}+\tau_{j}\tau_{i}+\tau_{i} \tau_{j} \tau_{i}+\tau_{j}\right)\left(1+\tau_{i_1-1}\right)\cdots \left(1+\tau_{i_2}\right) \\
        & = 1+ \sum_{k_{1}=i_1}^{j} \sum_{k_{2}=i_2}^{j} \tau_{\left(k_{1}, j, k_{2}\right)} = 1 + \tau_{i_{1}, j, i_{2}}
    \end{aligned}
    \end{equation*}
    
 \item If  $i=i_2$, then  $j=i_2+1$ (since $j-i=1)$ and  $i_1<i_2$, and then 
    
    \begin{equation*}
    \begin{aligned}
        \pi_{s_{i_1, j, i_2}} &=\pi_{s_{i_1}\cdots s_{i_2-1} s_{i_2} s_j s_{i_2}}  \\ 
        &= \left(1+\tau_{i_1}\right)\cdots \left(1+\tau_{i_2-1}\right)\left(1+\tau_{i}\right)\left(1+\tau_{j}\right)\left(1+\tau_{i}\right) \\ &=\left(1+\tau_{i_1}\right)\cdots \left(1+\tau_{i_2-1}\right)\left(1+\tau_{i} \tau_{j}+\tau_{j}\tau_{i}+\tau_{i} \tau_{j} \tau_{i}+\tau_{j}\right) \\
        & = 1+ \sum_{k_{1}=i_1}^{j} \sum_{k_{2}=i_2}^{j} \tau_{\left(k_{1}, j, k_{2}\right)} = 1 + \tau_{i_{1}, j, i_{2}}
    \end{aligned}
    \end{equation*}

\end{itemize}
 Assume by induction the theorem holds for $j-i=p>1$. Then $j=i+p$.

\begin{equation}\label{i,i+p,i}
  \left(\prod_{k=i}^{i+p-1} \pi_{s_{k}}\right) \pi_{s_{i+p}}\left(\prod_{k=i}^{i+p-1} \pi_{s_{i+p-k}}\right)=\pi_{s_{i, i+p, i}}=1+\sum_{k_{1}=i}^{i+p} \sum_{k_{2}=i}^{i+p} \tau_{\left(k_{1}, i+p, k_{2}\right)}=1+\tau_{i, i+p, i}
  \end{equation}
Now, we consider $j-i=p+1$.

\begin{equation}\label{i+1, i+p+1}
    \left(\prod_{k=i}^{i+p} s_{k}\right) s_{i+p+1}\left(\prod_{k=i}^{i+p} s_{i+p+1-k}\right)=  s_i \left(\prod_{k=i+1}^{i+p} s_{k}\right) s_{i+p+1}\left(\prod_{k=i+1}^{i+p} s_{i+p+1-k}\right) s_i.
\end{equation}
Since $(i+p+1)-(i+1)=p$, by Equation \eqref{i,i+p,i}, we get that Equation \eqref{i+1, i+p+1} is equivalent to: 

\begin{equation}\label{i+p+1}
\begin{aligned}
 &\left(\prod_{k=i}^{i+p} \pi_{s_{k}}\right) \pi_{s_{i+p+1}}\left(\prod_{k=i}^{i+p} \pi_{s_{i+p+1-k}}\right)= \pi_{s_i} \left(\prod_{k=i+1}^{i+p} \pi_{s_{k}}\right) \pi_{s_{i+p+1}}\left(\prod_{k=i+1}^{i+p} \pi_{s_{i+p+1-k}}\right) \pi_{s_i} \\ &= \left(1+\tau_i\right)\left(1+\tau_{i+1, i+p+1, i+1}\right)\left(1+\tau_i\right) = 1+\tau_{i, i+p+1, i}.
\end{aligned}
\end{equation}

Then, 
\begin{itemize}
    \item In case $i_2<i_1$: Multiplying Equation \eqref{i+p+1} by $\left(1+\tau_{i_1-1}\right)\left(1+\tau_{i_1-2}\right)\cdots \left(1+\tau_{i_2}\right)$ at the right, we get
    $$\pi_{s_{i, i+p+1, i_2}}=\left(1+\tau_{i, i+p+1, i}\right)\left(1+\tau_{i_1-1}\right)\left(1+\tau_{i_1-2}\right)\cdots \left(1+\tau_{i_2}\right)=1+\tau_{i, i+p+1, i_2}.$$
    \item In case $i_1<i_2$: Multiplying Equation \eqref{i+p+1} by  $\left(1+\tau_{i_1}\right)\left(1+\tau_{i_1+1}\right)\cdots \left(1+\tau_{i_2-1}\right)$ at the left, we get
    $$\pi_{s_{i_1, i+p+1, i}}=\left(1+\tau_{i_1}\right)\left(1+\tau_{i_1+1}\right)\cdots \left(1+\tau_{i_2-1}\right)\left(1+\tau_{i, i+p+1, i}\right)=1+\tau_{i_1, i+p+1, i}.$$ 
\end{itemize}

Hence, the proposition holds for every $1\leq i_1<j<i_2\leq n$.

\end{proof}

\begin{proposition}\label{tau-simply-laced}
Consider the mixed-sign Coxeter graph $\Gamma$ which is a line with vertices $s_1, s_2, \ldots, s_n$, which are signed by $f_1, f_2, \ldots, f_n$ respectively, such that for $1\leq i\leq n-1$, the vertices $s_i$ and $s_{i+1}$ are connected by a simply-laced edge, as it is described in Fig.1. For $1\leq i,~j\leq n$, let $E_{i,~j}$ be the matrix of $M_n(\mathbb{R})$ as it is defined in Definition \ref{E-i-j}. Then for $1\leq i<j\leq n$, the following holds:
\begin{itemize}
    \item $\prod_{k=i}^j \tau_{k}=E_{i,~j}\cdot \prod_{k=i}^{j-1}f_k\cdot \tau_j$
    \item $\tau_{\left(i,~j,~i\right)}=f_i\cdot f_j\cdot \tau_i$
 \item 
\begin{equation}\label{asec}
\begin{aligned} 
&\left[\left(1+\tau_{i}\right)\left(1+\tau_{i+1,~j,~i+1}\right)\right]^{r}-\left[\left(1+\tau_{i+1,~j,~i+1}\right)\left(1+\tau_{i}\right)\right]^{r}\\ &=\left(f_{i} f_{j}\right)^{r-1} \cdot \left\{\mathbb{E}_{2 r-1}{\left(1~,~ f_{i} f_{j}\right)}\right\} \cdot \left(\tau_{i}\cdot \tau_{i+1,~j,~i+1}-\tau_{i+1,~j,~i+1}\cdot \tau_{i}\right).
\end{aligned}
 \end{equation}
 \item 
\begin{equation}\label{asoc}
\begin{aligned} 
&[(1+\tau_{i})(1+\tau_{i+1,~j,~i+1})]^{r}(1+\tau_{i}) 
-[(1+\tau_{i+1,~j,~i+1})(1+\tau_{i})]^{r}(1+\tau_{i+1,~j,~i+1})
\\ &=\left(f_{i} f_{j}\right)^r \cdot \left\{\mathbb{E}_{2 r}{\left(1~, ~f_{i} f_{j}\right)}\right\} \cdot\left(\tau_{i}-\tau_{i+1,~j,~i+1}\right).
\end{aligned}
\end{equation}
\end{itemize}
\end{proposition}

\begin{proof}
The proof of the first part of the proposition  comes directly from Proposition \ref{tau-i1-ik}.
Hence, we turn to the proof of the second part of the proposition. 
$$\tau_{\left(i,~j,~i\right)}=f_i\cdot \prod_{k=i+1}^{j-1}f_k\cdot f_j\cdot \prod_{k=i+1}^{j-1}f_k\cdot \prod_{k=i}^{j-1}(E_{k,~k+1})^2\cdot \tau_i$$

Notice, $\left(E_{k,k+1}\right)^2=1$, and  $f_k\in \{+1, -1\}$ for every $i\leq k\leq j$. Therefore, $(f_k)^2=1$. Hence, we get $$\tau_{\left(i,~j,~i\right)}=f_i\cdot f_j\cdot \tau_i$$
\\

Now, we turn to the third and the forth part of the proposition.
First, notice, by Definition \ref{tau-i,j,i}, 
\begin{equation}\label{tau-i+1-j-i+1}
    \tau_{i+1, ~j,~i+1}:=
   \sum_{k_{1}=i+1}^{j} \sum_{k_{2}=i+1}^{j} \tau_{\left(k_{1}, ~j, ~k_{2}\right)}.
\end{equation}

Hence,
\begin{equation}\label{asenew}
\begin{aligned}
&\left[\left(1+\tau_{i}\right)\left(1+\tau_{i+1,~j,~i+1}\right)\right]^{r}-\left[\left(1+\tau_{i+1,~j,~i+1}\right)\left(1+\tau_{i}\right)\right]^{r}\\ &=\left[\left(1+\tau_{i}\right)\left(1+\sum_{k_{1}=i+1}^{j} \sum_{k_{2}=i+1}^{j} \tau_{\left(k_{1}, ~j, ~k_{2}\right)}\right)\right]^{r}-\left[\left(1+\sum_{k_{1}=i+1}^{j} \sum_{k_{2}=i+1}^{j} \tau_{\left(k_{1}, ~j, ~k_{2}\right)}\right)\left(1+\tau_{i}\right)\right]^{r}
\end{aligned}
\end{equation}

and

\begin{equation}\label{asonew}
\begin{aligned}
&\left[\left(1+\tau_{i}\right)\left(1+\tau_{i+1,~j,~i+1}\right)\right]^{r}\left(1+\tau_{i}\right)-\left[\left(1+\tau_{i+1,~j,~i+1}\right)\left(1+\tau_{i}\right)\right]^{r}\left(1+\tau_{i+1,~j,~i+1}\right)
\\ &=\left[\left(1+\tau_{i}\right)\left(1+\sum_{k_{1}=i+1}^{j} \sum_{k_{2}=i+1}^{j} \tau_{\left(k_{1}, ~j, ~k_{2}\right)}\right)\right]^{r}\left(1+\tau_{i}\right) \\&-\left[\left(1+\sum_{k_{1}=i+1}^{j} \sum_{k_{2}=i+1}^{j} \tau_{\left(k_{1}, ~j, ~k_{2}\right)}\right)\left(1+\tau_{i}\right)\right]^{r}\left(1+\sum_{k_{1}=i+1}^{j} \sum_{k_{2}=i+1}^{j} \tau_{\left(k_{1}, ~j, ~k_{2}\right)}\right)
\end{aligned}
\end{equation}

Now, notice that,

$\left[\left(1+\tau_{i}\right)\left(1+\sum_{k_{1}=i+1}^{j} \sum_{k_{2}=i+1}^{j} \tau_{\left(k_{1}, ~j, ~k_{2}\right)}\right)\right]^{r}$, ~$\left[\left(1+\sum_{k_{1}=i+1}^{j} \sum_{k_{2}=i+1}^{j} \tau_{\left(k_{1}, ~j, ~k_{2}\right)}\right)\left(1+\tau_{i}\right)\right]^{r}$, \\ \\ $\left[\left(1+\tau_{i}\right)\left(1+\sum_{k_{1}=i+1}^{j} \sum_{k_{2}=i+1}^{j} \tau_{\left(k_{1}, ~j, ~k_{2}\right)}\right)\right]^{r}\left(1+\tau_{i}\right)$, ~and \\ \\ $\left[\left(1+\sum_{k_{1}=i+1}^{j} \sum_{k_{2}=i+1}^{j} \tau_{\left(k_{1}, ~j, ~k_{2}\right)}\right)\left(1+\tau_{i}\right)\right]^{r}\left(1+\sum_{k_{1}=i+1}^{j} \sum_{k_{2}=i+1}^{j} \tau_{\left(k_{1}, ~j, ~k_{2}\right)}\right)$ \\ \\ are sum of elements of the form:
\begin{itemize}
    \item $(\tau_i\cdot \sum_{k_{1}=i+1}^{j} \sum_{k_{2}=i+1}^{j} \tau_{\left(k_{1}, ~j, ~k_{2}\right)})^k$;
    \item $(\tau_i\cdot \sum_{k_{1}=i+1}^{j} \sum_{k_{2}=i+1}^{j} \tau_{\left(k_{1}, ~j, ~k_{2}\right)})^k\cdot \tau_i$;
    \item $(\sum_{k_{1}=i+1}^{j} \sum_{k_{2}=i+1}^{j} \tau_{\left(k_{1}, ~j, ~k_{2}\right)}\cdot \tau_i)^k$;
    \item $(\sum_{k_{1}=i+1}^{j} \sum_{k_{2}=i+1}^{j} \tau_{\left(k_{1}, ~j, ~k_{2}\right)}\cdot \tau_i)^k\cdot \sum_{k_{1}=i+1}^{j} \sum_{k_{2}=i+1}^{j} \tau_{\left(k_{1}, ~j, ~k_{2}\right)}$.
\end{itemize}
   for $0\leq k\leq r$.
Now, by Proposition \ref{tau-property}, $\tau_i\tau_{i'}=0$ for $i'\notin\{i-1, i, i+1\}$, the following holds:
\begin{itemize}
    \item $(\tau_i\cdot \sum_{k_{1}=i+1}^{j} \sum_{k_{2}=i+1}^{j} \tau_{\left(k_{1}, ~j, ~k_{2}\right)})^k=(\tau_i\cdot \tau_{\left(i+1, ~j, ~i+1\right)})^{k-1}\cdot\tau_i\cdot \sum_{k_{1}=i+1}^{j} \sum_{k_{2}=i+1}^{j} \tau_{\left(k_{1}, ~j, ~k_{2}\right)}$;
    \item $(\tau_i\cdot \sum_{k_{1}=i+1}^{j} \sum_{k_{2}=i+1}^{j} \tau_{\left(k_{1}, ~j, ~k_{2}\right)})^k\cdot \tau_i=(\tau_i\cdot \tau_{\left(i+1, ~j, ~i+1\right)})^k\cdot\tau_i$;
    \item $(\sum_{k_{1}=i+1}^{j} \sum_{k_{2}=i+1}^{j} \tau_{\left(k_{1}, ~j, ~k_{2}\right)}\cdot \tau_i)^k=\sum_{k_{1}=i+1}^{j} \sum_{k_{2}=i+1}^{j} \tau_{\left(k_{1}, ~j, ~k_{2}\right)}\cdot(\tau_i\cdot \tau_{\left(i+1, ~j, ~i+1\right)})^{k-1}\cdot\tau_i $;
    \item $(\sum_{k_{1}=i+1}^{j} \sum_{k_{2}=i+1}^{j} \tau_{\left(k_{1}, ~j, ~k_{2}\right)}\cdot \tau_i)^k\cdot \sum_{k_{1}=i+1}^{j} \sum_{k_{2}=i+1}^{j} \tau_{\left(k_{1}, ~j, ~k_{2}\right)}=\sum_{k_{1}=i+1}^{j} \sum_{k_{2}=i+1}^{j} \tau_{\left(k_{1}, ~j, ~k_{2}\right)}\cdot(\tau_i\cdot \tau_{\left(i+1, ~j, ~i+1\right)})^{k-1}\cdot\tau_i\cdot \sum_{k_{1}=i+1}^{j} \sum_{k_{2}=i+1}^{j} \tau_{\left(k_{1}, ~j, ~k_{2}\right)} $.
\end{itemize}

Now, by Definition \ref{tau-i,j,i}, $\tau_{i}\cdot \tau_{\left(i+1, ~j, ~i+1\right)}\cdot \tau_{i}=\tau_ {\left(i, ~j, ~i\right)}$, where by the second part of the proposition $\tau_ {\left(i, ~j, ~i\right)}=f_i\cdot f_j\cdot \tau_{i}$ . 
Therefore, by  replacing $\tau_{\left(i+1, ~j, ~i+1\right)} $ instead of $\tau_j$ on Proposition  \ref{r-dickson}, and applying the same process of  Proposition  \ref{r-dickson}, we get the desired results of the proposition.

\end{proof}

\begin{theorem}\label{line-3-relation}
Let ~$\Gamma$ be the mixed-sign Coxeter graph which is a line with $n$ vertices $s_1, s_2, \ldots, s_n$ which are signed by $f_1, f_2, \ldots, f_n$ respectively, such that for $1\leq i\leq n-1$ the vertex $s_i$ is connected to $s_{i+1}$ by a simply-laced edge, as it is described in Fig 1. Then, for $1\leq i<j\leq n$, apart from the Coxeter relations:

\begin{align}\label{3com}
 (s_i\cdot s_j)^2=1 \ \Leftrightarrow j-i\geq 2.   
\end{align}

\begin{align}\label{3cox}
(s_i\cdot s_{i+1})^3=1 \ \Leftrightarrow f_i=f_{i+1}.    
\end{align}

which hold by Proposition \ref{2-+}, the following generalized Coxeter relations hold in $W(\Gamma)$:
\begin{align} \label{2}
(s_i\cdot s_{i+1, ~j, ~i+1})^3=1\quad \text{if and only if}\quad f_i=f_j \end{align}
\end{theorem}

\begin{proof}


By Definition \ref{tau}, $\pi_{s_i}=1+\tau_i$ and by Proposition \ref{tau-simply-laced}, $\pi_{s_{i+1, j, i+1}}=1+\tau_{i+1, j, i+1}$.
\\

Hence,
\begin{equation}\label{s-i-r}
  (s_i\cdot s_{i+1, ~j, ~i+1})^r=1
\end{equation}

is equivalent to 

\begin{equation}\label{tau-i-r}
  \left[\left(1+\tau_{i}\right)\left(1+\tau_{i+1,~j,~i+1}\right)\right]^{r}=1 
\end{equation}

Since, both $(1+\tau_i)$ and $(1+\tau_{i+1,~j,~i+1})$ are involutions,    Equation \eqref{tau-i-r} is equivalent to:

\begin{itemize}
    \item In case of an even $r$:
    \begin{equation}\label{tau-i-r2}
        \left[\left(1+\tau_{i}\right)\left(1+\tau_{i+1,~j,~i+1}\right)\right]^{\frac{r}{2}}-\left[\left(1+\tau_{i+1,~j,~i+1}\right)\left(1+\tau_{i}\right)\right]^{\frac{r}{2}};
    \end{equation}
    \item In case of an odd $r$:
    \begin{equation}\label{tau-i-r3}
        \left[\left(1+\tau_{i}\right)\left(1+\tau_{i+1,~j,~i+1}\right)\right]^{\frac{r-1}{2}}\left(1+\tau_{i}\right)-\left[\left(1+\tau_{i+1,~j,~i+1}\right)\left(1+\tau_{i}\right)\right]^{\frac{r-1}{2}}\left(1+\tau_{i+1,~j,~i+1}\right).
    \end{equation}
\end{itemize}
Now, by Proposition \ref{tau-simply-laced} the following holds:
\begin{itemize}
\item In case of an even $r$:
    \begin{equation}\label{asec1}
\begin{aligned} 
&\left[\left(1+\tau_{i}\right)\left(1+\tau_{i+1,~j,~i+1}\right)\right]^{\frac{r}{2}}-\left[\left(1+\tau_{i+1,~j,~i+1}\right)\left(1+\tau_{i}\right)\right]^{\frac{r}{2}}\\ &=\left(f_{i} f_{j}\right)^{\frac{r}{2}-1} \cdot \left\{\mathbb{E}_{r-1}{\left(1~,~ f_{i} f_{j}\right)}\right\} \cdot \left(\tau_{i}\cdot \tau_{i+1,~j,~i+1}-\tau_{i+1,~j,~i+1}\cdot \tau_{i}\right).
\end{aligned}
 \end{equation}
 \item  In case of an odd $r$:
\begin{equation}\label{asoc1}
\begin{aligned} 
&[(1+\tau_{i})(1+\tau_{i+1,~j,~i+1})]^{\frac{r-1}{2}}(1+\tau_{i}) 
-[(1+\tau_{i+1,~j,~i+1})(1+\tau_{i})]^{\frac{r-1}{2}}(1+\tau_{i+1,~j,~i+1})
\\ &=\left(f_{i} f_{j}\right)^{\frac{r-1}{2}} \cdot \left\{\mathbb{E}_{r-1}{\left(1~, ~f_{i} f_{j}\right)}\right\} \cdot\left(\tau_{i}-\tau_{i+1,~j,~i+1}\right).
\end{aligned}
\end{equation}
\end{itemize}

Notice, by Proposition \ref{sum-tau-i-beginning},  $\tau_{i}\cdot \tau_{i+1,~j,~i+1}-\tau_{i+1,~j,~i+1}\cdot \tau_{i}\neq 0$ and $\tau_{i}-\tau_{i+1,~j,~i+1}\neq 0$.
Therefore, Equation \eqref{s-i-r} holds if and only if
\begin{equation}\label{dickson-zero}
 \mathbb{E}_{r-1}{\left(1~,~ f_{i} f_{j}\right)}=0.   
\end{equation}

Since, $f_i, f_j$ is either $+1$ or $-1$,  Equation \eqref{dickson-zero} holds if and only if

$$r=3\quad\quad \text{and} ~~f_i=f_j.$$

\end{proof}

\begin{remark}\label{only-relation-line}
 Notice, in the proof of Theorem \ref{line-3-relation}, apart from the usual Coxeter relations $\left(s_{i}s_{j}\right)^r=1$, where $r\in\{2,3\}$ in case of a simply-laced mixed-sign Coxeter group and \\$1\leq i, j\leq n$ (notice, in case of $r=3$ the signs of the vertices $s_i$ and $s_j$ satisfy $f_i=f_j$),  we consider only relations of the form $\left(s_{i}s_{i+1,~j,~i+1}\right)^r=1$ (where in the case of mixed-sign Coxeter group with an associated mixed-sign Coxeter graph which is a simply-laced line, $r$ equals to $3$). By a similar way as it has shown that every relation in a Coxeter group is derived from a relation of a form $\left(s_{i}s_{j}\right)^r=1$ for some $r$ \cite{BOOK}, it can be shown that any other type of relation (which can be found by the same method of the proof of Theorem \ref{line-3-relation}) is derived from the mentioned relations in Theorem \ref{line-3-relation}.
\end{remark}

\begin{example}
Consider the following mixed-sign Coxeter graph $\Gamma$:
\begin{center}
	\includegraphics[width=.40\textwidth]{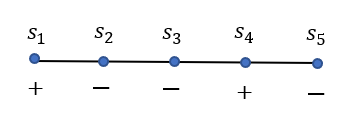}
	\end{center}
	
Then
\begin{equation}\label{example-simply-laced-line}
    \begin{aligned}
    W(\Gamma)=\langle s_1, s_2, s_3, s_4, s_5 ~|~ &s_i^2=1 ~ \text{for} ~ 1\leq i\leq 5, ~(s_2s_3)^3=(s_1s_2s_3s_4s_3s_2)^3=(s_3s_4s_5s_4)^3=1, \\ 
    & (s_is_j)^2=1 ~ \text{for} ~ 1\leq i<j\leq 5 ~ \text{and} ~j-i\geq 2\rangle.
    \end{aligned}
\end{equation}

As it is mentioned in Remark \ref{only-relation-line}, there are other relations in $W(\Gamma)$ which are derived from the described relations 
in Equation \eqref{example-simply-laced-line}. For example, it can be shown that the following relation holds:  

\begin{equation}\label{derived-relation}
 \left(\left[s_5 s_4 s_3 s_4 s_5\right]\cdot \left[s_4 s_3 s_2 s_3 s_4\right]\right)^2=1.   
\end{equation}

Notice,

\begin{equation}\label{3-4-5-4}
 \left(s_3s_4s_5s_4\right)^3=1,    
\end{equation}
which appears in the presentation of $W(\Gamma)$. By conjugating Equation \eqref{3-4-5-4} by $s_4$, the equation is equivalent to:

\begin{equation}\label{3-4-5-power3}
  \left(s_4s_3s_4s_5\right)^3=1,    
\end{equation}
 Equation \eqref{3-4-5-power3} is equivalent to 

\begin{equation}\label{5-4-3-4-5}
    s_5\left(s_4 s_3 s_4\right) s_5 = \left(s_4 s_3 s_4\right) s_5 \left(s_4 s_3 s_4\right).
\end{equation}

Notice, $\left(s_2 s_4\right)^2$ implies $s_2s_4 = s_4 s_2$, which implies $s_4 s_2 s_4= s_2$. Hence,  we have

\begin{equation}\label{2-4-3-4-2}
  s_4 s_3 s_2 s_3 s_4 =   \left(s_4 s_3 s_4\right) s_2 \left(s_4 s_3 s_4\right).
\end{equation}

Hence, by applying the results of Equations \eqref{5-4-3-4-5} and \eqref{2-4-3-4-2} in Equation \eqref{derived-relation}, one obtain

\begin{equation}\label{5-2-commute}
   \left(\left[ \left(s_4 s_3 s_4\right) s_5 \left(s_4 s_3 s_4\right)\right]\left[\left(s_4 s_3 s_4\right) s_2 \left(s_4 s_3 s_4\right)\right] \right)^2=1, 
\end{equation}
where by conjugating  Equation \eqref{5-2-commute} by $s_4 s_3 s_4$, the equation is equivalent to 
\begin{equation}
    \left(s_5 s_2\right)^2=1.
\end{equation}

\end{example}

\subsection{Mixed-sign Coxeter graph is a simply-laced simple cycle}\label{simply laced}

In this subsection we we give a description of mixed-sign Coxeter groups in terms of generators and relations, where  considering mixed-sign Coxeter groups $W(\Gamma)$, such that  $\Gamma$ is a simply-laced simple cycle with vertices $s_1, s_2, \ldots, s_n$, which are signed by $f_1, f_2, \ldots, f_n$ respectively, where the vertices $s_i$ and $s_{i+1}$ are connected by a simply-laced edge for $1\leq i\leq n-1$, and additionally, the vertices $s_{n}$ and $s_1$ are connected also by a simply-laced edge.

\begin{center}
	\includegraphics[width=.27\textwidth]{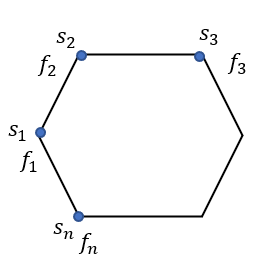}
\end{center}
	\begin{center}
	
{\scriptsize { \normalsize $Fig.\ 2.$}}
 	    \end{center}

\begin{remark}\label{n+1=1}
For a convenience, in this subsection  we consider $n+1$ to be equal to $1$ and $1-1$ to be equal to $n$ (i.e., cyclic addition modulo $n$),  concerning indices of vertices  of a simple cycle which are connected by an edge. Then, we can say, vertices $s_i$ and $s_{i+1}$ are connected by an edge for $1\leq i\leq n$ (without considering $s_n$ and $s_1$ as a special case).
\end{remark}

\begin{proposition}\label{independent-simply-cycle}
Let $\Gamma$ be a mixed-sign Coxeter graph, a simple cycle with $n$ vertices, $s_1, s_2, \ldots, s_n$ which are signed by $f_1, f_2, \ldots, f_n$ respectively, such that  $s_i$ and $s_{i+1}$ are connected by a simply-laced edge for $1\leq i\leq n$, as it shown in Fig. 2. Let $\nu_{i}$ be a vector in $\mathbb{R}^n$ which we get by considering the $i$-th row of $\tau_i$ as a vector for every $1\leq i\leq n$ (as it is defined in Definition \ref{nu}). Then, the associated bilinear form is non-degenerate, which is equivalent to:
\\

The set of vectors 
$$\{\nu_{1}, \nu_{2}, \ldots, \nu_{n}\}$$
is a linearly independent set of vectors in $\mathbb{R}^n$, unless one of the following conditions holds
\begin{itemize}
    \item $f_i=+1$ for every $1\leq i\leq n$;
    \item $f_i=-1$ for every $1\leq i\leq n$ and $n$ is even.
\end{itemize}
\end{proposition}

\begin{proof}
Since for every $1\leq i\leq n$,  $\nu_{i}$ is a vector in $\mathbb{R}^n$ by considering the $i$-th row of $\tau_i$ as a vector, $\nu_i$ has the following form ($n+1$ is considered to be $1$ and $1-1$ is considered to be $n$ as it is mentioned in Remark \ref{n+1=1}):
\begin{itemize}
    \item $\left(\nu_{i}\right)_j=0$ ~in case $j\notin\{i-1, ~i, ~i+1\}$;
    \item $\left(\nu_{i}\right)_j=f_i$ ~in case $j=i-1$ or $j=i+1$;
    \item $\left(\nu_{i}\right)_j=-2$ ~in case $j=i$.
\end{itemize}
Since either $f_i=+1$ or $f_i=-1$, it can been shown easily that the set of vectors $$\{\nu_{1}, \nu_{2}, \ldots, \nu_{n}\}$$
is a linearly independent set of vectors in $\mathbb{R}^n$ unless $f_i=+1$ for every $1\leq i\leq n$ or $f_i=-1$ for every $1\leq i\leq n$ and $n$ is even.
\end{proof}

\begin{remark}\label{s-tilde-non-independent}
As a conclusion of Proposition \ref{independent-simply-cycle}, the  affine Coxeter groups  $\widetilde{S}_n$  are the only cases of mixed-sign simply-laced Coxeter groups with associated mixed-sign Coxeter graph is a simple cycle, and with a degenerate associated bilinear form.
\end{remark}

\begin{proposition}\label{tau-simply-laced-cycle}

Let $\Gamma$ be a simply-laced simple cycle with vertices $s_1, s_2, \ldots, s_n$  such that the vertex $s_i$ is signed by $f_i$ for  $1\leq i\leq n$, and the vertices $s_i$ and $s_{i+1}$ are connected by a simply-laced edge for $1\leq i\leq n$, where by Remark \ref{n+1=1}, we consider $n+1$ to be $1$ and $1-1$ to be $n$. For $1\leq i,j\leq n$, such that $i\neq j$, let define $$\kappa_{i,j}:=f_{i+1}\cdot f_{i+2}\cdots f_{j-1}+f_{i-1}\cdot f_{i-2}\cdots f_{j+1}$$

Then the following holds:

\begin{itemize}

\item 
\begin{equation}\label{asec-cycle}
\begin{aligned} 
&\left[\left(1+\tau_{j+1,~i,~j+1}\right)\left(1+\tau_{i+1,~j,~i+1}\right)\right]^{r}-\left[\left(1+\tau_{i+1,~j,~i+1}\right)\left(1+\tau_{j+1,~i,~j+1}\right)\right]^{r}\\ &=\left(f_{i}f_{j}\right)^{r-1} \cdot \left\{\mathbb{E}_{2 r-1}{\left(\kappa_{i,j}~,~f_{i}f_{j} \right)}\right\} \cdot \left(\tau_{j+1, ~i, ~, j+1}\cdot \tau_{i+1,~j,~i+1}-\tau_{i+1, ~j, ~, i+1}\cdot \tau_{j+1,~i,~j+1}\right).
\end{aligned}
 \end{equation}
 \item 
\begin{equation}\label{asoc-cycle}
\begin{aligned} 
&[(1+\tau_{j+1,~i,~j+1})(1+\tau_{i+1,~j,~i+1})]^{r}(1+\tau_{j+1,~i,~j+1}) \\ &
-[(1+\tau_{i+1,~j,~i+1})(1+\tau_{j+1,~i,~j+1})]^{r}(1+\tau_{i+1,~j,~i+1})
\\ &=\left(f_{i}f_{j}\right)^r \cdot \left\{\mathbb{E}_{2 r}{\left(\kappa_{i,j}~, ~f_{i}f_{j}\right)}\right\} \cdot\left(\tau_{j+1,~i,~j+1}-\tau_{i+1,~j,~i+1}\right).
\end{aligned}
\end{equation}
\end{itemize}

\end{proposition}

\begin{proof}
by Definition \ref{tau-i,j,i}, 
\begin{equation}\label{tau-i+1-j-i+1-j+1-i-j+1}
    \tau_{i+1, ~j,~i+1}:=
   \sum_{k_{1}=i+1}^{j} \sum_{k_{2}=i+1}^{j} \tau_{\left(k_{1}, ~j, ~k_{2}\right)}\quad \quad \tau_{j+1, ~i,~j+1}:=
   \sum_{k_{1}=j+1}^{i} \sum_{k_{2}=j+1}^{i} \tau_{\left(k_{1}, ~i, ~k_{2}\right)}.
\end{equation}

Hence,
\begin{equation}\label{asenew-cycle}
\begin{aligned}
&\left[\left(1+\tau_{j+1,~i,~j+1}\right)\left(1+\tau_{i+1,~j,~i+1}\right)\right]^{r}-\left[\left(1+\tau_{i+1,~j,~i+1}\right)\left(1+\tau_{j+1,~i,~j+1}\right)\right]^{r}\\ &=\left[\left(1+\sum_{k_{1}=j+1}^{i} \sum_{k_{2}=j+1}^{i} \tau_{\left(k_{1}, ~i, ~k_{2}\right)}\right)\left(1+\sum_{k_{1}=i+1}^{j} \sum_{k_{2}=i+1}^{j} \tau_{\left(k_{1}, ~j, ~k_{2}\right)}\right)\right]^{r}\\&-\left[\left(1+\sum_{k_{1}=i+1}^{j} \sum_{k_{2}=i+1}^{j} \tau_{\left(k_{1}, ~j, ~k_{2}\right)}\right)\left(1+\sum_{k_{1}=j+1}^{i} \sum_{k_{2}=j+1}^{i} \tau_{\left(k_{1}, ~i, ~k_{2}\right)}\right)\right]^{r}
\end{aligned}
\end{equation}

and

\begin{equation}\label{asonew-cycle}
\begin{aligned}
&\left[\left(1+\tau_{j+1,~i,~j+1}\right)\left(1+\tau_{i+1,~j,~i+1}\right)\right]^{r}\left(1+\tau_{j+1,~i,~j+1}\right)-\left[\left(1+\tau_{i+1,~j,~i+1}\right)\left(1+\tau_{j+1,~i,~j+1}\right)\right]^{r}\left(1+\tau_{i+1,~j,~i+1}\right)
\\ &=\left[\left(1+\sum_{k_{1}=j+1}^{i} \sum_{k_{2}=j+1}^{i} \tau_{\left(k_{1}, ~i, ~k_{2}\right)}\right)\left(1+\sum_{k_{1}=i+1}^{j} \sum_{k_{2}=i+1}^{j} \tau_{\left(k_{1}, ~j, ~k_{2}\right)}\right)\right]^{r}\left(1+\sum_{k_{1}=j+1}^{i} \sum_{k_{2}=j+1}^{i} \tau_{\left(k_{1}, ~i, ~k_{2}\right)}\right) \\&-\left[\left(1+\sum_{k_{1}=i+1}^{j} \sum_{k_{2}=i+1}^{j} \tau_{\left(k_{1}, ~j, ~k_{2}\right)}\right)\left(1+\sum_{k_{1}=j+1}^{i} \sum_{k_{2}=j+1}^{i} \tau_{\left(k_{1}, ~i, ~k_{2}\right)}\right)\right]^{r}\left(1+\sum_{k_{1}=i+1}^{j} \sum_{k_{2}=i+1}^{j} \tau_{\left(k_{1}, ~j, ~k_{2}\right)}\right)
\end{aligned}
\end{equation}

Now, notice that,

$\left[\left(1+\sum_{k_{1}=j+1}^{i} \sum_{k_{2}=j+1}^{i} \tau_{\left(k_{1}, ~i, ~k_{2}\right)}\right)\left(1+\sum_{k_{1}=i+1}^{j} \sum_{k_{2}=i+1}^{j} \tau_{\left(k_{1}, ~j, ~k_{2}\right)}\right)\right]^{r}$, \\ \\ $\left[\left(1+\sum_{k_{1}=i+1}^{j} \sum_{k_{2}=i+1}^{j} \tau_{\left(k_{1}, ~j, ~k_{2}\right)}\right)\left(1+\tau_{\sum_{k_{1}=j+1}^{i} \sum_{k_{2}=j+1}^{i} \tau_{\left(k_{1}, ~i, ~k_{2}\right)}}\right)\right]^{r}$, \\ \\ $\left[\left(1+\sum_{k_{1}=j+1}^{i} \sum_{k_{2}=j+1}^{i} \tau_{\left(k_{1}, ~i, ~k_{2}\right)}\right)\left(1+\sum_{k_{1}=i+1}^{j} \sum_{k_{2}=i+1}^{j} \tau_{\left(k_{1}, ~j, ~k_{2}\right)}\right)\right]^{r}\left(1+\sum_{k_{1}=j+1}^{i} \sum_{k_{2}=j+1}^{i} \tau_{\left(k_{1}, ~i, ~k_{2}\right)}\right)$, ~and \\ \\ $\left[\left(1+\sum_{k_{1}=i+1}^{j} \sum_{k_{2}=i+1}^{j} \tau_{\left(k_{1}, ~j, ~k_{2}\right)}\right)\left(1+\sum_{k_{1}=j+1}^{i} \sum_{k_{2}=j+1}^{i} \tau_{\left(k_{1}, ~i, ~k_{2}\right)}\right)\right]^{r}\left(1+\sum_{k_{1}=i+1}^{j} \sum_{k_{2}=i+1}^{j} \tau_{\left(k_{1}, ~j, ~k_{2}\right)}\right)$ \\ \\ are sum of elements of the form:
\begin{itemize}
    \item $(\sum_{k_{1}=j+1}^{i} \sum_{k_{2}=j+1}^{i} \tau_{\left(k_{1}, ~i, ~k_{2}\right)}\cdot \sum_{k_{1}=i+1}^{j} \sum_{k_{2}=i+1}^{j} \tau_{\left(k_{1}, ~j, ~k_{2}\right)})^k$;
    \item $(\sum_{k_{1}=j+1}^{i} \sum_{k_{2}=j+1}^{i} \tau_{\left(k_{1}, ~i, ~k_{2}\right)}\cdot \sum_{k_{1}=i+1}^{j} \sum_{k_{2}=i+1}^{j} \tau_{\left(k_{1}, ~j, ~k_{2}\right)})^k\cdot \sum_{k_{1}=j+1}^{i} \sum_{k_{2}=j+1}^{i} \tau_{\left(k_{1}, ~i, ~k_{2}\right)}$;
    \item $(\sum_{k_{1}=i+1}^{j} \sum_{k_{2}=i+1}^{j} \tau_{\left(k_{1}, ~j, ~k_{2}\right)}\cdot \sum_{k_{1}=j+1}^{i} \sum_{k_{2}=j+1}^{i} \tau_{\left(k_{1}, ~i, ~k_{2}\right)})^k$;
    \item $(\sum_{k_{1}=i+1}^{j} \sum_{k_{2}=i+1}^{j} \tau_{\left(k_{1}, ~j, ~k_{2}\right)}\cdot\sum_{k_{1}=j+1}^{i} \sum_{k_{2}=j+1}^{i} \tau_{\left(k_{1}, ~i, ~k_{2}\right)} )^k\cdot \sum_{k_{1}=i+1}^{j} \sum_{k_{2}=i+1}^{j} \tau_{\left(k_{1}, ~j, ~k_{2}\right)}$.
\end{itemize}
   for $0\leq k\leq r$.
Now, by Proposition \ref{tau-property}, $\tau_i\tau_{i'}=0$ and $\tau_j\tau_{j'}=0$ for $i'\notin\{i-1, i, i+1\}$, and $j'\notin\{j-1, j, j+1\}$. Hence,  the following holds:
\begin{itemize}
  \item 
  \begin{equation}
      \begin{aligned}
      &\sum_{k_{1}=j+1}^{i} \sum_{k_{2}=j+1}^{i} \tau_{\left(k_{1}, ~i, ~k_{2}\right)}\cdot \sum_{k_{1}=i+1}^{j} \sum_{k_{2}=i+1}^{j} \tau_{\left(k_{1}, ~j, ~k_{2}\right)})^k \\&=\sum_{k_{1}=j+1}^{i-1}\tau_{\left(k_{1}, ~i-1, ~i-1\right)} \\ &\cdot \left[ \tau_i \left(\tau_{i+1} \tau_{i+2}\cdots \tau_{j-2}\tau_{j-1}+ \tau_{i-1} \tau_{i-2}\cdots \tau_{j+2}\tau_{j+1}\right) \tau_{j} \left(\tau_{i+1}\tau_{i+2}\cdots \tau_{j-2}\tau_{j-1}+ \tau_{i-1} \tau_{i-2}\cdots \tau_{j+2} \tau_{j+1}\right)\right]^{k-1} \\ & \cdot \tau_i \left(\tau_{i+1}\tau_{i+2}\cdots \tau_{j-2} \tau_{j-1}+ \tau_{i-1}\tau_{i-2}\cdots \tau_{j+2} \tau_{j+1}\right) \tau_{j}\cdot \sum_{k_{2}=i+1}^{j-1} \tau_{\left(j-1, ~j-1, ~k_{2}\right)};
      \end{aligned}
  \end{equation}
  
  \item 
  \begin{equation}
      \begin{aligned}
      &\sum_{k_{1}=i+1}^{j} \sum_{k_{2}=i+1}^{j} \tau_{\left(k_{1}, ~j, ~k_{2}\right)}\cdot \sum_{k_{1}=j+1}^{i} \sum_{k_{2}=j+1}^{i} \tau_{\left(k_{1}, ~i, ~k_{2}\right)})^k \\&=\sum_{k_{1}=i+1}^{j-1}\tau_{\left(k_{1}, ~j-1, ~j-1\right)} \\ &\cdot \left[ \tau_j \left(\tau_{j+1} \tau_{j+2}\cdots \tau_{i-2}\tau_{i-1}+ \tau_{j-1} \tau_{j-2}\cdots \tau_{i+2}\tau_{i+1}\right) \tau_{i} \left(\tau_{j+1}\tau_{j+2}\cdots \tau_{i-2}\tau_{i-1}+ \tau_{j-1} \tau_{j-2}\cdots \tau_{i+2} \tau_{i+1}\right)\right]^{k-1} \\ & \cdot \tau_j \left(\tau_{j+1}\tau_{j+2}\cdots \tau_{i-2} \tau_{i-1}+ \tau_{j-1}\tau_{j-2}\cdots \tau_{i+2} \tau_{i+1}\right) \tau_{i}\cdot \sum_{k_{2}=j+1}^{i-1} \tau_{\left(i-1, ~i-1, ~k_{2}\right)};
      \end{aligned}
  \end{equation}
  
  \item 
  \begin{equation}
      \begin{aligned}
      &\sum_{k_{1}=j+1}^{i} \sum_{k_{2}=j+1}^{i} \tau_{\left(k_{1}, ~i, ~k_{2}\right)}\cdot \sum_{k_{1}=i+1}^{j} \sum_{k_{2}=i+1}^{j} \tau_{\left(k_{1}, ~j, ~k_{2}\right)})^k\cdot \sum_{k_{1}=j+1}^{i} \sum_{k_{2}=j+1}^{i} \tau_{\left(k_{1}, ~i, ~k_{2}\right)} \\&=\sum_{k_{1}=j+1}^{i-1}\tau_{\left(k_{1}, ~i-1, ~i-1\right)} \\ &\cdot \left[ \tau_i \left(\tau_{i+1} \tau_{i+2}\cdots \tau_{j-2}\tau_{j-1}+ \tau_{i-1} \tau_{i-2}\cdots \tau_{j+2}\tau_{j+1}\right) \tau_{j} \left(\tau_{i+1}\tau_{i+2}\cdots \tau_{j-2}\tau_{j-1}+ \tau_{i-1} \tau_{i-2}\cdots \tau_{j+2} \tau_{j+1}\right)\right]^{k} \\ & \cdot \tau_i\cdot \sum_{k_{2}=j+1}^{i-1} \tau_{\left(i-1, ~i-1, ~k_{2}\right)};
      \end{aligned}
  \end{equation}
  
  \item 
  \begin{equation}
      \begin{aligned}
      &\sum_{k_{1}=i+1}^{j} \sum_{k_{2}=i+1}^{j} \tau_{\left(k_{1}, ~j, ~k_{2}\right)}\cdot \sum_{k_{1}=j+1}^{i} \sum_{k_{2}=j+1}^{i} \tau_{\left(k_{1}, ~i, ~k_{2}\right)})^k\cdot \sum_{k_{1}=i+1}^{j} \sum_{k_{2}=i+1}^{j} \tau_{\left(k_{1}, ~j, ~k_{2}\right)} \\&=\sum_{k_{1}=i+1}^{j-1}\tau_{\left(k_{1}, ~j-1, ~j-1\right)} \\ &\cdot \left[ \tau_j \left(\tau_{j+1} \tau_{j+2}\cdots \tau_{i-2}\tau_{i-1}+ \tau_{j-1} \tau_{j-2}\cdots \tau_{i+2}\tau_{i+1}\right) \tau_{i} \left(\tau_{j+1}\tau_{j+2}\cdots \tau_{i-2}\tau_{i-1}+ \tau_{j-1} \tau_{j-2}\cdots \tau_{i+2} \tau_{i+1}\right)\right]^{k} \\ & \cdot \tau_j\cdot \sum_{k_{2}=i+1}^{j-1} \tau_{\left(j-1, ~j-1, ~k_{2}\right)};
      \end{aligned}
  \end{equation}

\end{itemize}

Then, by  replacing $\tau_{\left(i+1, ~j, ~i+1\right)} $ instead of $\tau_j$ and by  replacing $\tau_{\left(j+1, ~i, ~j+1\right)}$ instead of $\tau_i$ on Proposition  \ref{r-dickson}, and applying the same process of  Proposition  \ref{r-dickson}, we get the desired results of the proposition.

\end{proof}

Now, we give the theorem, which describes the relation on a mixed-sign simply-laced Coxeter group where the associated mixed-sign Coxeter graph is a simple cycle.

\begin{theorem}\label{mixed-sign-s-laced-cycle}
Let $\Gamma$ be a simply-laced simple cycle with vertices $s_1, s_2, \ldots, s_n$  such that the vertex $s_i$ is signed by $f_i$ for  $1\leq i\leq n$, and the vertices $s_i$ and $s_{i+1}$ are connected by a simply-laced edge for $1\leq i\leq n$, where by Remark \ref{n+1=1}, we consider $n+1$ to be $1$ and $1-1$ to be $n$. Then for $1\leq i,~j\leq n$, apart from the Coxeter relations:

\begin{align}\label{4com}
 (s_i\cdot s_j)^2=1 \ \Leftrightarrow 2\leq |i-j|\leq n-2.   
\end{align}

\begin{align}\label{4cox}
(s_i\cdot s_{i+1})^3=1 \ \Leftrightarrow f_i=f_{i+1}.    
\end{align}

which hold by Proposition \ref{2-+}, the following generalized Coxeter relations hold in $W(\Gamma)$:

\begin{align} \label{4l}
(s_{i}\cdot s_{i+1,~j,~i+1})^3=1 \ \Leftrightarrow \ j\neq i-1\quad \text{and}\quad f_i=f_j 
\end{align}

\begin{align} \label{4}
(s_{j+1, i, j+1}\cdot s_{i+1, j, i+1})^2=1 \ \Leftrightarrow \ \frac{\prod_{k=1}^n f_k}{f_i\cdot f_j} =-1
.\end{align}
\end{theorem}

\begin{proof}
Notice, the subgraph of $\Gamma$ with the set of vertices $\{s_i, s_{i+1}, \ldots, s_j\}$ is a line for $j\neq i-1$. Therefore, the relations from Equation \eqref{4l} hold  by Theorem \ref{line-3-relation}, where considering the mixed-sign Coxeter group associated to that subgraph.
Hence, we turn to the proof of the relation from Equation \eqref{4}.

First, notice by Proposition \ref{independent-simply-cycle}, the only cases where the associated bilinear form is degenerate, may happen, just in case where one of the following holds:
\begin{itemize}
  \item $f_i=+1$ for every  $1\leq i\leq n$;
    \item $f_i=-1$ for every $1\leq i\leq n$ and $n$ is even.
\end{itemize}
In both cases there is no $i$ and $j$ such that $1\leq i, j\leq n$ and the condition $\frac{\prod_{k=1}^n f_k}{f_i\cdot f_j} =-1$ of 
Equation \eqref{4} holds. Hence, in the proof of the theorem we may assume and use that the associated bilinear form is non-degenerate.
\\

By Definition \ref{tau-i,j,i}, $\pi_{s_{i+1, j, ~i+1}}=1+\tau_{i+1, j, ~i+1}$ and $\pi_{s_{j+1, i,~j+1}}=1+\tau_{j+1, i, ~j+1}$. \\

Hence, we suppose to prove: 
$$[(1+\tau_{i+1, j,~i+1})(1+\tau_{j+1, i,~j+1})]^2=1 \ \ \  \text{if and only if} \ \ \ \frac{\prod_{k=1}^n f_k}{f_i\cdot f_j} =-1.$$\\

Notice, the condition
\begin{equation}[(1+\tau_{i+1, j, i+1})(1+\tau_{j+1, i, j+1})]^2=1\end{equation}
is equivalent to
\begin{equation}\label{i-j-i=j-i-j}
   (1+\tau_{i+1, j, i+1})(1+\tau_{j+1, i, j+1})=(1+\tau_{j+1, i, j+1})(1+\tau_{i+1, j, i+1}), \end{equation}
where, Equation \eqref{i-j-i=j-i-j} is equivalent to:
\begin{equation}\label{tau-i-j-i=j-i-j}
    \tau_{i+1, j, i+1}\cdot \tau_{j+1, i, j+1} = \tau_{j+1, i, j+1}\cdot \tau_{i+1, j, i+1} 
\end{equation}

By the definition of $\tau_{i+1,~j,~i+1}$ and $\tau_{j+1,i,~j+1}$ as it is defined in Definition \ref{tau-i,j,i}, and then by Proposition \ref{sum-tau-i-beginning}, Equation \eqref{tau-i-j-i=j-i-j} holds if and only if 
\begin{equation}\label{i-j-cycle-commute}
\tau_{i+1, j, i+1}\cdot \tau_{j+1, i, j+1} = 0\quad \text{and}\quad \tau_{j+1, i, j+1}\cdot \tau_{i+1, j, i+1}=0.
\end{equation}

Since by Proposition \ref{independent-simply-cycle}, the associated bilinear form is non-degenerate (apart from the case of $W(\Gamma)=\widetilde{S}_n$, which is not relevant for our proof as it is mentioned at the beginning of the proof), we apply Proposition \ref{sum-tau-i-finishing} on Equation \eqref{i-j-cycle-commute}.
Thus for every $p\in\{i+1, i+2, \ldots, j-1, j\}$ and $q\in \{j+1, j+2, \ldots, i-1, i\}$,
\begin{equation}\label{tau-i-tau-j=0}
    \left(\sum_{k_1=i+1}^{j}\tau_{(p, j, k_1)}\right)\cdot \left(\sum_{k_2=j+1}^{i}\tau_{(k_2, i, q)}\right)=0.
\end{equation}

Notice, by second part of Proposition \ref{tau-property},  if $|k_1-k_2|\neq 1$, then $\tau_{(p, j, k_1)}\cdot \tau_{(k_2, i, q)}=0$. \\

Therefore,  Equation \eqref{tau-i-tau-j=0} is equivalent to:
\begin{equation}\label{k1-k2}
    \tau_{(p, j, j)}\cdot \tau_{(j+1, i, q)} + \tau_{(p, j, i+1)}\cdot \tau_{(i, i, q)} =0
\end{equation}
where, Equation \eqref{k1-k2} can be written in the following form:
\begin{equation}\label{i+j=i-j=0}
    \tau_{(p, j-1, j-1)}\cdot \left(\tau_{j}\cdot \tau_{j+1}\cdots \tau_{i-1}\tau_{i} + \tau_{j}\cdot \tau_{j-1}\cdots \tau_{i+1}\tau_{i} \right)\cdot \tau_{(i-1, i-1, q)}=0.
\end{equation}

 By Proposition \ref{sum-tau-i-finishing}, we deduce
from Equation \eqref{i+j=i-j=0} the following observation:
\begin{equation}\label{i-j-i-j-i}
  \tau_{j}\cdot \tau_{j+1}\cdots \tau_{i-1}\tau_{i} + \tau_{j}\cdot \tau_{j-1}\cdots \tau_{i+1}\tau_{i}=0.  
\end{equation}

By the third part of Proposition \ref{tau-property} Equation \eqref{i-j-i-j-i} may occur if and only if:

\begin{equation}\label{f-i-f-j=0}
    f_j\cdot f_{j+1}\cdots f_{i-1}+f_j\cdot f_{j-1}\cdots f_{i+1}=
    f_j\cdot\left(f_{j+1}\cdots f_{i-1}+ f_{j-1}\cdots f_{i+1}\right)=0.
\end{equation}

Since $f_k\in \{1, -1\}$ for every $1\leq k\leq n$, Equation \eqref{f-i-f-j=0} holds if and only if 
$$f_{j+1}\cdot f_{j+2}\cdots f_{i-1}=-f_{j-1}\cdot f_{j-2}\cdots f_{i+1}  $$
which is equivalent to:
$$\frac{\prod_{k=1}^{n}f_k}{f_i\cdot f_j}=-1.$$

Now, we find which relations of a form

\begin{equation}\label{s-i-r-cycle}
  (s_{j+1, ~i, ~j+1}\cdot s_{i+1, ~j, ~i+1})^r=1\quad\quad \text{for}~~ r\geq 3.
\end{equation}

may hold in $W$.

Consider the relation from Equatuion \eqref{s-i-r-cycle}.

Since $s_{j+1, ~i, ~j+1}$ and $ s_{i+1, ~j, ~i+1}$ are both involutions,  Equation \eqref{s-i-r-cycle} is equivalent to:
\begin{itemize}
    \item In case of an even $r$:
    \begin{equation}\label{tau-i-r2-cycle}
        \left[\left(1+\tau_{j+1,~i,~j+1}\right)\left(1+\tau_{i+1,~j,~i+1}\right)\right]^{\frac{r}{2}}-\left[\left(1+\tau_{i+1,~j,~i+1}\right)\left(1+\tau_{j+1,~i,~j+1}\right)\right]^{\frac{r}{2}};
    \end{equation}
    \item In case of an odd $r$:
    \begin{equation}\label{tau-i-r3-cycle}
    \begin{aligned}
 &\left[\left(1+\tau_{j+1,~i,~j+1}\right)\left(1+\tau_{i+1,~j,~i+1}\right)\right]^{\frac{r-1}{2}}\left(1+\tau_{j+1,~i,~j+1}\right)\\&-\left[\left(1+\tau_{i+1,~j,~i+1}\right)\left(1+\tau_{j+1,~i,~j+1}\right)\right]^{\frac{r-1}{2}}\left(1+\tau_{i+1,~j,~i+1}\right).
    \end{aligned}
        \end{equation}
\end{itemize}

Now, by Proposition \ref{tau-simply-laced-cycle} the following holds:
\begin{itemize}
\item In case of an even $r$:
    \begin{equation}\label{asec1-cycle}
\begin{aligned} 
&\left[\left(1+\tau_{j+1,~i,~j+1}\right)\left(1+\tau_{i+1,~j,~i+1}\right)\right]^{\frac{r}{2}}-\left[\left(1+\tau_{i+1,~j,~i+1}\right)\left(1+\tau_{j+1,~i,~j+1}\right)\right]^{\frac{r}{2}}\\ &=\left(f_{i} f_{j}\right)^{\frac{r}{2}-1} \cdot \left\{\mathbb{E}_{r-1}{\left(\kappa_{i,j}~,~ f_{i} f_{j}\right)}\right\} \cdot \left(\tau_{j+1,~i,~j+1}\cdot \tau_{i+1,~j,~i+1}-\tau_{i+1,~j,~i+1}\cdot \tau_{j+1,~i,~j+1}\right).
\end{aligned}
 \end{equation}
 \item  In case of an odd $r$:
\begin{equation}\label{asoc1-cycle}
\begin{aligned} 
&[(1+\tau_{j+1,~i,~j+1})(1+\tau_{i+1,~j,~i+1})]^{\frac{r-1}{2}}(1+\tau_{j+1,~i,~j+1}) \\ &
-[(1+\tau_{i+1,~j,~i+1})(1+\tau_{j+1,~i,~j+1})]^{\frac{r-1}{2}}(1+\tau_{i+1,~j,~i+1})
\\ &=\left(f_{i} f_{j}\right)^{\frac{r-1}{2}} \cdot \left\{\mathbb{E}_{r-1}{\left(\kappa_{i,j}~, ~f_{i} f_{j}\right)}\right\} \cdot\left(\tau_{j+1,~i,~j+1}-\tau_{i+1,~j,~i+1}\right).
\end{aligned}
\end{equation}
where, $\kappa_{i,j}:=f_{i+1}\cdot f_{i+2}\cdots f_{j-1}+f_{i-1}\cdot f_{i-2}\cdots f_{j+1}$.
\end{itemize}

Notice, 
$$\left(\tau_{j+1,~i,~j+1}-\tau_{i+1,~j,~i+1}\right)\neq 0$$ 
and  $$\left(\tau_{j+1,~i,~j+1}\cdot \tau_{i+1,~j,~i+1}-\tau_{i+1,~j,~i+1}\cdot \tau_{j+1,~i,~j+1}\right)= 0\Leftrightarrow s_{j+1,~i,~j+1}\cdot s_{i+1,~j,~i+1}=s_{i+1,~j,~i+1}\cdot s_{j+1,~i,~j+1}$$  .
which holds if and only if $r=2$ and $\frac{\prod_{k=1}^{n}f_k}{f_i f_j}=-1$, as it has been already proved in the theorem.

Since, $f_i, f_j$ is either $+1$ or $-1$, ~$\kappa_{i,j}\in\{-2, ~0, ~2\}$ by using Equations \eqref{asec1-cycle} and \eqref{asoc1-cycle}, Equation \eqref{s-i-r-cycle} holds if and only if 

\begin{equation}\label{dickson-cycle}
   \mathbb{E}_{r-1}{\left(\kappa_{i,j}~, ~f_{i} f_{j}\right)}=0. 
\end{equation}

where, Equation \eqref{dickson-cycle} is satisfied if and only if 
the following conditions hold:
\begin{itemize}
    \item $r=2m$, where $m\in \mathbb{N}$;
    \item $\kappa_{i,j}=f_{i+1}\cdot f_{i+2}\cdots f_{j-1}+f_{i-1}\cdot f_{i-2}\cdots f_{j+1}=0$, which is equivalent to $\frac{\prod_{k=1}^n f_k}{f_i\cdot f_j} =-1$.
\end{itemize}

Hence,

the defining relation is 

\begin{align} 
(s_{j+1, i, j+1}\cdot s_{i+1, j, i+1})^2=1 \ \Leftrightarrow \ \frac{\prod_{k=1}^n f_k}{f_i\cdot f_j} =-1
.\end{align}

\end{proof}

\begin{remark}\label{only-relation-cycle}
 Notice, in the proof of Theorem \ref{mixed-sign-s-laced-cycle}, apart from the usual Coxeter relations $\left(s_{i}s_{j}\right)^r=1$, where $r\in\{2,3\}$ in case of a simply-laced mixed-sign Coxeter group and  $1\leq~i, j\leq n$ (notice, in case of $r=3$ the signs of the vertices $s_i$ and $s_j$ satisfy $f_i=f_j$),  we consider only relations of the form  $\left(s_{i}s_{i+1,~j,~i+1}\right)^{r_1}=1$ for $j\neq i-1$ (which is induced by the subgraph with the set of  vertices  $\{s_i, s_{i+1}, \ldots, s_j\}$) and relations of the form  $\left([s_{j+1, i, j+1}][s_{i+1,~j,~i+1}]\right)^{r_2}=1$ (where the relation contains all the generators $s_k$ for $1\leq k\leq n$ of $W(\Gamma)$). Notice,  in the case of mixed-sign Coxeter group with an associated mixed-sign Coxeter graph which is a simply-laced simple cycle, $r_1$ equals to $3$ and  $r_2$ equals to $2$.  By a similar way as it has shown that every relation in a Coxeter group is derived from a relation of a form $\left(s_{i}s_{j}\right)^r=1$ for some $r$ \cite{BOOK}, it can be shown that any other type of relation (which can be found by the same method of the proof of Theorems \ref{line-3-relation} and \ref{mixed-sign-s-laced-cycle}) is derived from the mentioned relations in Theorem \ref{mixed-sign-s-laced-cycle}.
\end{remark}

\bigskip

\begin{example}

 The presentation of the mixed-sign Coxeter group which associate mixed-sign Coxeter graph is the following square:
\begin{center}
	\includegraphics[width=.26\textwidth]{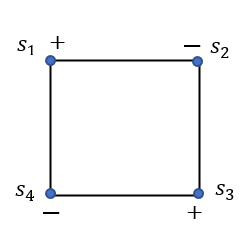}
\end{center}
is:

\begin{align*}
&\langle s_{1}, s_{2},s_{3}, s_{4}  |
s_{1}^{2}=s_{2}^{2}=s_{3}^{2}=s_{4}^{2}=1, \ (s_{1} s_{3})^2=(s_2s_4)^2=1, \\
&(s_{1} s_{2} s_{3} s_{2})^{3}=
(s_{2} s_{3} s_{4} s_{3})^{3}=
(s_{3} s_{4} s_{1} s_{4})^{3}=
(s_{4} s_{1} s_{2} s_{1})^{3}=1, \\
&(s_{1} s_{2} s_{3} s_{4} s_{3} s_{2})^{2}=
(s_{2} s_{3} s_{4} s_{1} s_{4} s_{3})^{2}=
(s_{3} s_{4} s_{1} s_{2} s_{1} s_{4})^{2}= \\
&(s_{4} s_{1} s_{2} s_{3} s_{2} s_{1})^{2}=1 \rangle.
\end{align*}
\end{example}


\subsubsection{Classification of some mixed-sign simply laced simple cycles}

Now, we give a classification of mixed-sign Coxeter groups, where the associated mixed-sign Coxeter graph is a simply-laced simple cycle for some special cases. We start with a full classification of the cases where the associated simple cycle has $3$ vertices, then we turn to the cases of simple cycles, where all the vertices of the cycle are signed by the same sign. Finally, we prove that the mixed-sign Coxeter group which we get where the associated mixed-sign Coxeter graph, which is a simple cycle, has an odd length, and the vertices of the graph are signed by $-1$, is the Coxeter $D_n$ (It has been mentioned in \cite{Japan Thesis} without a proof, since it has been shown there by a GAP check only). We start with a useful lemma, which we use in the proofs.

\begin{lemma}\label{affine-conjugate}
Let $\Gamma$ be a mixed-sign Coxeter graph which is a simply-laced simple cycle of length $n$, with the vertices $s_1, s_2, \ldots, s_n$, which are signed by $f_1, f_2, \ldots, f_n$ respectively, such that $f_{i-1}=f_i=f_{i+1}$, for some $1\leq i\leq n$ (By Remark \ref{n+1=1}, ~$n+1:=1$ and $1-1:=n$). Then  the following element of $W(\Gamma)$: $$s_i(s_{i+1}s_{i+2}\cdots s_{i-2}s_{i-1}s_{i-2}\cdots s_{i+2}s_{i+1})$$
is conjugate to 
$$s_{i+1}(s_{i+2} s_{i+3}\cdots s_{i-2} s_{i-1}  s_i s_{i-1} s_{i-2}\cdots s_{i+3} s_{i+2}).$$

\medskip

In particular, if $f_i$ is the same for all $1\leq i\leq n$ (i.e., either $f_i=-1$ or $f_i=+1$ for all $1\leq i\leq n$), then all the elements of the form $$s_i(s_{i+1}s_{i+2}\cdots s_{i-2}s_{i-1}s_{i-2}\cdots s_{i+2}s_{i+1})$$ are conjugate to
$$s_1(s_2 s_3\cdots s_{n-1}s_n s_{n-1}\cdots s_3 s_2).$$ 

\end{lemma}

\begin{proof}
Consider $W(\Gamma)$ with the generators $s_1, s_2, \ldots s_n$, such that $f_{i-1}=f_i=f_{i+1}$, for some  $1\leq i\leq n$. Then the following relations holds in $W(\Gamma)$:
\begin{itemize}
    \item $s_i^2=1$;
    \item $s_i s_{i+1} s_i = s_{i+1} s_i s_{i+1}$ where we consider by Remark \ref{n+1=1}, $n+1:=1$;
    \item $s_i s_j = s_j s_i$ for $|i-j|\neq 1$.
\end{itemize}
Therefore,
\begin{align*}
& (s_i s_{i+1}) [s_i(s_{i+1}s_{i+2}\cdots s_{i-2}s_{i-1}s_{i-2}\cdots s_{i+2}s_{i+1})] (s_{i+1} s_i) \\
& =s_i s_i s_{i+1} s_i s_{i+2} s_{i+3}\cdots s_{i-2}s_{i-1}s_{i-2}\cdots s_{i+2} s_i \\
 & =s_{i+1} ( s_i s_{i+2} s_{i+3}\cdots s_{i-2}s_{i-1}s_{i-2}\cdots s_{i+3} s_{i+2} s_i ) \\
 &  =s_{i+1} ( s_{i+2} s_{i+3}\cdots s_{i-2} s_i s_{i-1} s_i s_{i-2}\cdots s_{i+3} s_{i+2}) \\
 & =s_{i+1} ( s_{i+2} s_{i+3}\cdots s_{i-2} s_{i-1}  s_i s_{i-1} s_{i-2}\cdots s_{i+3} s_{i+2}).
\end{align*}
\end{proof}

\textbf{The case: $n=3$}
\\

Now, we classify the mixed-sign Coxeter groups which mixed-sign Coxeter graph is a simply-laced simple cycle of length $3$. We show a strong connection to mixed-sign Coxeter groups which associated mixed-sign Coxeter graph is a line with 3 vertices.

\begin{enumerate}
    \item Consider the mixed-sign Coxeter graph (which is a Coxeter graph, since all the labels are +1) at the right hand side:
     \begin{center}
\includegraphics[width=.70\textwidth]{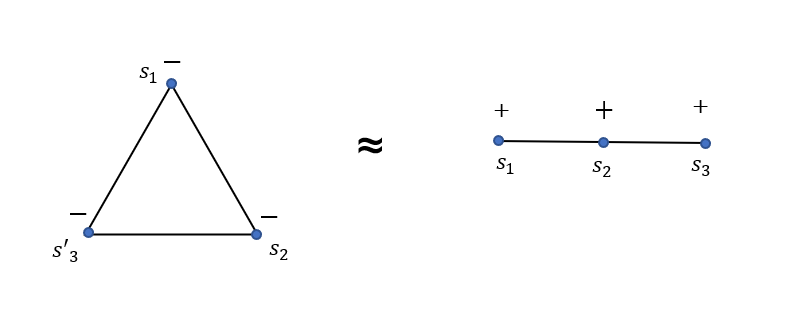}
\end{center}
    The associated Coxeter group, which is the symmetric group $S_4$, has the following presentation:
    $$\langle s_{1}, s_{2},s_{3} \ | \ s_{1}^{2}=s_{2}^{2}=s_{3}^{2}=1, \
(s_{1} s_{3})^2=1, \
(s_{1} s_{2})^3=(s_2s_3)^3=1  \rangle.$$ 

Define $s'_3:=s_2s_3s_2.$ Then the presentation of the same  Coxeter group by using generators $s_1, s_2, s'_3$ is as follow:
$$\langle s_{1}, s_{2},s'_{3} \ | \
s_{1}^{2}=s_{2}^{2}={s'}_{3}^{2}=1, 
\ (s_{1}s_{2}s'_{3}s_{2})^2=(s_{1} s_{2})^3=(s_2s'_3)^3=(s_1s'_3)^3=1 \rangle.$$
Where, the presentation is same to the presentation of the mixed-sign Coxeter group associated to the mixed-sign Coxeter graph which is a simple cycle at the left hand side.

    
    Notice,  the three elements $s_2s'_3s_1s'_3$, $s_1 s_2 s'_3 s_2$, and $s'_3 s_1 s_2 s_1$ are conjugate elements in $W(\Gamma)$ by Lemma \ref{affine-conjugate}, therefore $(s_1 s_2 s'_3 s_2)^2=1$ implies $(s_2s'_3s_1s'_3)^2=1$  and $(s'_3 s_1 s_2 s_1)^2=1$ as well.
  
 \item Consider the mixed-sign Coxeter graph  at the right hand side: 
 \begin{center}
\includegraphics[width=.70\textwidth]{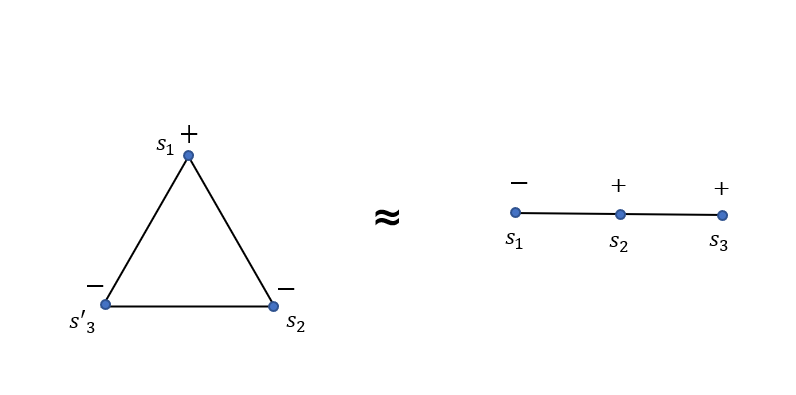}
\end{center}
The associated mixed-sign Coxeter group has the following presentation:
    $$\langle s_{1}, s_{2},s_{3} \ | \ s_{1}^{2}=s_{2}^{2}=s_{3}^{2}=1, \
({s}_{1} s_{3})^2=1, \
(s_2s_3)^3=1  \rangle.$$ 
Define $s'_3:=s_2s_3s_2.$ Then the presentation of the same mixed-sign Coxeter group by using generators $s_1, s_2, s'_3$ is as follow:
$$\langle s_{1}, s_{2},{s'}_{3} \ | \
s_{1}^{2}=s_{2}^{2}={s'}_{3}^{2}=1, 
\ (s_{1}s_{2}{s'}_{3}s_{2})^2=(s_2 s'_3)^3=1 \rangle.$$
Where, the presentation is same to the presentation of the mixed-sign Coxeter group associated to the mixed-sign Coxeter graph which is a simple cycle at the left hand side.


    \item Consider the mixed-sign Coxeter graph at the right hand side:
     \begin{center}
\includegraphics[width=.68\textwidth]{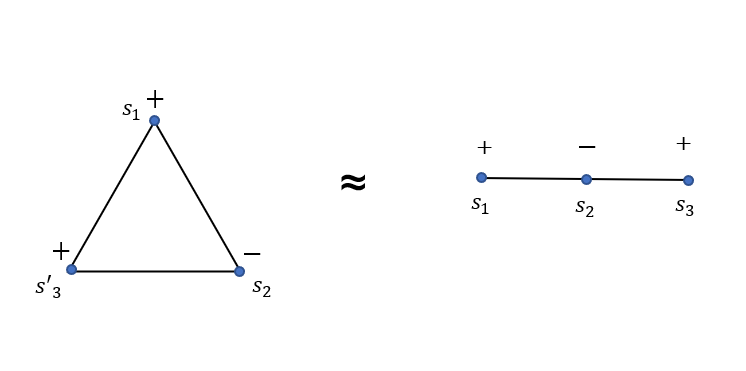}
\end{center}
The associated mixed-sign Coxeter group has the following presentation:
   $$\langle s_{1}, s_{2},s_{3} \ | \ s_{1}^{2}=s_{2}^{2}=s_{3}^{2}=1, \
(s_{1} s_{3})^2=1, \
(s_1s_2s_3s_2)^3=1  \rangle.$$ 
Define $s'_3:=s_2s_3s_2.$ Then the presentation of the same mixed-sign Coxeter group by using generators $s_1, s_2, s'_3$ is as follow:
$$\langle s_{1}, s_{2},s'_{3} \ | \
s_{1}^{2}=s_{2}^{2}={s'}_{3}^{2}=1, 
\ (s_{1}s_{2}s'_{3}s_{2})^2=1, \
(s_1s'_3)^3=1 \rangle.$$
Where, the presentation is same to the presentation of the mixed-sign Coxeter group associated to the mixed-sign Coxeter graph which is a simple cycle at the left hand side.

The mixed-sign Coxeter groups presentations which corresponds to the above mixed-sign Coxeter graphs are \textbf{not} Coxter groups since there are relations in the groups which involves more then two generators.

\end{enumerate}







\textbf{Simply-laced simple cycle - all the vertices are signed by the same sign}
\\

Now, we give a classifications of the mixed-sign Coxeter groups, where the associated mixed-sign Coxeter graph is a simply-laced simple cycle, where all the vertices are signed by the same sign (either +1 or -1). 

\begin{proposition}\label{Classification-Cyclic graph}

Let $\Gamma$ be a mixed-sign Coxeter graph, such that $\Gamma$ is a simple cycle of length~$n$ (i.e., simple cycle with $n$ vertices and $n$ edges).
\begin{itemize}

    \item  In case all the vertices of $\Gamma$ are signed by +1, $$W(\Gamma)\approx \widetilde{S}_{n} .$$
    
    
    \item  In case all the vertices of $\Gamma$ are signed by -1, we have the following subcases:
\begin{enumerate}
    \item In case $\Gamma$ is an even length cycle (i.e., $n$ is even), then
    $$W(\Gamma)\approx\widetilde{S}_{n}.$$
    \item In case  $\Gamma$ is an odd length cycle (i.e., $n$ is odd), then
    $$W(\Gamma)\approx D_n.$$
\end{enumerate}
\end{itemize}

\end{proposition}
\begin{proof}
The case where all the vertices of the mixed-sign Coxeter graph is signed by $+1$, the associated mixed-sign Coxeter group is a Coxeter group. Hence by the classification of Coxeter groups \cite{BOOK} $W(\Gamma)\approx\widetilde{S}_{n}$. 
Hence, assume all the vertices of the mixed-sign Coxeter graph is signed by $-1$. 

Assume case 1. Then,  $\Gamma$ is an even length cycle, and therefore, $\Gamma$ is a bipartite graph. Then by Remark \ref{MC+=MC-}, $ W(\Gamma)\approx W(-\Gamma)$. Hence we have:
    $$W(\Gamma)\approx W(-\Gamma)\approx \widetilde{S}_{n}.$$
Assume case 2.
Let $\Gamma=C_{2n+1}^{-}$ be a mixed-sign Coxeter graph which is a simply-laced simple cycle with the vertices $s_1, s_2, \ldots, s_{2n+1}$ (i.e., the number of vertices is odd), and $f_i=-1$ for all $1\leq i\leq 2n+1$. Then, by Theorem  \ref{mixed-sign-s-laced-cycle}, the following relations holds in addition to the standard Coxeter relations 
\begin{equation}\label{cicle-i-2n+1-relation}
[s_i(s_{i+1} s_{i+2} \cdots s_{i-2} s_{i-1} s_{i-2}\cdots  s_{i+1})]^2=1,
\end{equation}
for $1\leq i\leq 2n+1$, (where by Remark \ref{n+1=1},  $(2n+1) + 1 := 1$).
By Lemma \ref{affine-conjugate}, all the element of the form 
$$s_i(s_{i+1} s_{i+2} \cdots s_{i-2} s_{i-1} s_{i-2}\cdots  s_{i+1} s_i) $$
are conjugate elements in $W(C_{2n+1}^{-})$. Hence, all the relations in Equation \eqref{cicle-i-2n+1-relation} can be deduced from the following relation.
\begin{equation}\label{cicle-2n+1-relation}
[s_1(s_2 s_3 \cdots s_{2n} s_{2n+1} s_{2n}\cdots  s_3 s_2)]^2=1.
\end{equation}
Define $s'_{0}$ to be $(s_2 s_3 \cdots s_{2n} s_{2n+1} s_{2n}\cdots  s_3 s_2)$.
Notice, the elements $s_1, s_2, \ldots s_{2n}, s'_{0}$ generates the group $W(C_{2n+1}^{-})$.
Now, consider the presentation of $W(C_{2n+1}^{-})$ presented by the generators $s_1, s_2, \ldots s_{2n}, s'_{0}$. Then, the following relations holds:
By the definition of $s'_{0}$ and Equation \eqref{cicle-2n+1-relation}, we have:  $$(s_1s'_{0})^2=1.$$

Now, consider the relations  of $s'_{0}$ with $s_i$ for $2\leq i\leq 2n$.
By using the following properties:
 The mixed-sign Coxeter group associated to the subgraph with the vertices $s_2, s_3, \ldots, s_{2n+1}$ is the Coxeter group $S_{2n+1}$, with the standard relations 
    \begin{itemize}
        \item $(s_i s_{i+1})^3=1$;
        \item $(s_i s_j)^2=1$ for $|i-j|>1$.
    \end{itemize}
     
Therefore, for every $2\leq i<j\leq 2n$, the following holds: $$s_i s_{i+1}\cdots s_{j-1} s_{j} s_{j-1}\cdots s_{i+1} s_{i}  = s_{j} s_{j-1} \cdots s_{i+1} s_{i} s_{i+1}\cdots s_{j-1} s_{j}.$$

Hence,

\begin{align*}
\left(s'_{0} s_2\right)^3 &= \left([s_2 s_3 \cdots s_{2n} s_{2n+1} s_{2n}\cdots  s_3 s_2]s_2\right)^3= \left([s_3 \cdots s_{2n} s_{2n+1} s_{2n}\cdots  s_3]s_2\right)^3 \\ 
&= \left([s_{2n+1} s_{2n} \cdots s_4 s_3 s_4\cdots s_{2n} s_{2n+1}]s_2\right)^3=\left(s_3 s_2\right)^3=1.
\end{align*}

Now, for $3\leq k\leq 2n$:
\begin{align*}
  \left(s'_{0} s_k\right)^2 &= \left([s_2 s_3 \cdots s_{2n} s_{2n+1} s_{2n}\cdots  s_3 s_2]s_k\right)^2 = \left([s_{k-1} s_{k} s_{k+1} s_{k} s_{k-1}] s_k\right)^2 \\
  &= \left([ s_{k} s_{k+1} s_{k} ] [s_{k-1} s_k s_{k-1}]\right)^2 = \left([ s_{k} s_{k+1} s_{k} ] [s_{k} s_{k-1} s_{k}]\right)^2 =\left( s_{k+1} s_{k-1} \right)^2=1
\end{align*}

\end{proof}

\begin{example}
Consider the graph

\begin{center}
	\includegraphics[width=.28\textwidth]{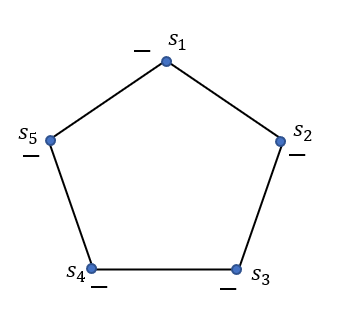}
\end{center}

Where $n=5$ and $f_i=-1 \ \forall i$, then the relations are:

\begin{equation}\label{C5-}
\begin{aligned}
&W(C_{5}^{-})=\langle s_1, s_2, s_3, s_4, s_5~|~{s_1}^2={s_2}^2={s_3}^2={s_4}^2={s_5}^2=1, \\ &(s_1s_3)^2=(s_1s_4)^2=(s_2s_4)^2=(s_2s_5)^2=(s_3s_5)^2=1, \\
&(s_1s_2)^3=(s_2s_3)^3=(s_3s_4)^3=(s_4s_5)^3=(s_1s_5)^3=1, \\ &(s_1[s_2s_3s_4s_5s_4s_3s_2])^2=1. \rangle
\end{aligned}
\end{equation}

Define $s'_{0}$ to be $s_2s_3s_4s_5s_4s_3s_2$, and consider the presentation of 
$W(C_{5}^{-})$ by the generators $s_1, s_2, s_3, s_4, s'_{0}$.
Then the relations in Equation \eqref{C5-} become to be as follow:
\begin{itemize}
    \item ${s_5}^2=1$ becomes to be ${s'_{0}}^2=1$;
    \item $(s_2s_5)^2=(s_3s_5)^2=(s_4s_5)^3=1$ become to be 
    $(s_3s'_{0})^2=(s_4s'_{0})^2=(s_2s'_{0})^3=1$;
    \item $(s_1[s_2s_3s_4s_5s_4s_3s_2])^2=1$ becomes to be $(s_1s'_{0})^2=1$;
    \item The relation $(s_1s_5)^3$ is derived by the rest of the relations of $W(C_{5}^{-})$ as follow:
    The definition $s'_{0}:=s_2s_3s_4s_5s_4s_3s_2$ implies 
    $s_5=s_4s_3s_2s'_{0}s_2s_3s_4$.
    Then, by using $s_1s_4=s_4s_1$, ~$s_1s_3=s_3s_1$,  ~$s_2s'_{0}s_2=s'_{0}s_2s'_{0}$,  ~$s'_{0}s_1=s_1s'_{0}$, and $(s_1s_2)^3=1$, the following holds:
    $$\begin{aligned}
        &(s_1s_5)^3=(s_1[s_4s_3s_2s'_{0}s_2s_3s_4])^3\\&=(s_1[s_3s_2s'_{0}s_2s_3])^3= (s_1[s_2s'_{0}s_2])^3\\&=(s_1[s'_{0}s_2s'_{0}])^3=(s_1s_2)^3=1.
    \end{aligned}$$
\end{itemize}

Hence, the presentation of $W(C_{5}^{-})$ by using the set of generators $\{s_1, s_2, s_3, s_4, s'_{0}\}$ instead of the set $\{s_1, s_2, s_3, s_4, s_5\}$ is a Coxeter group presentation, where the associated Coxeter graph is as follow:

\begin{center}
	\includegraphics[width=.43\textwidth]{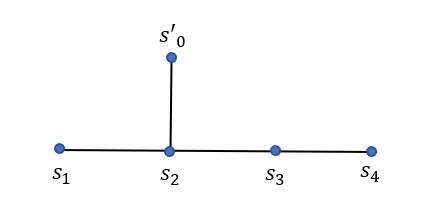}
\end{center}

\end{example}
\section{Conclusions and future plans}\label{future plans}

 We gave a  classification of the quotients of Coxeter groups in terms of generators and relations, which we get as a simply-laced mixed-sign Coxeter group \cite{HIR11, Japan Thesis}, for a Coxeter graph which is  either a line or a simple cycle, with a condition of non-singularity of the matrix of the associated bilinear form. It is interesting to generalize the results of the paper for further simply-laced Coxeter graphs (e.g. general trees, or non-simple cycles), and for non-simply-laced mixed-sign Coxeter graphs.

\end{document}